\documentclass[11pt]{amsart}
\usepackage{amsmath,amsthm,amssymb,xypic,array}
\usepackage[T1]{fontenc}
\usepackage{amsfonts}
\usepackage{amsmath}
\usepackage{amssymb}
\usepackage{amsthm}
\usepackage{setspace}
\usepackage[cp850]{inputenc}
\usepackage[bookmarksnumbered,plainpages,hypertex]{hyperref}
\usepackage{graphicx}
\usepackage{fancyhdr}

\usepackage{tikz}
\usepackage{booktabs}
\usepackage{longtable}
\usepackage{geometry}
\usetikzlibrary{trees}

\newcommand{\ndo}{\node[draw,circle,inner sep=1]}

\providecommand{\U}[1]{\protect\rule{.1in}{.1in}}
\providecommand{\U}[1]{\protect\rule{.1in}{.1in}}
\providecommand{\U}[1]{\protect\rule{.1in}{.1in}}
\providecommand{\U}[1]{\protect\rule{.1in}{.1in}}
\providecommand{\U}[1]{\protect\rule{.1in}{.1in}}
\input{xy}
\xyoption{all}
\theoremstyle{theorem}

\newtheorem{Theorem}{Theorem}[section]
\newtheorem*{theoremn}{Theorem}
\newtheorem*{question}{Question}

\newtheorem{Lemma}[Theorem]{Lemma}
\newtheorem{Proposition}[Theorem]{Proposition}
\newtheorem{Corollary}[Theorem]{Corollary}

\theoremstyle{definition}

\newtheorem{Definition}[Theorem]{Definition}
\newtheorem{Remark}[Theorem]{Remark}

\numberwithin{equation}{section}

\newcommand{\arXiv}[1]{\href{http://arxiv.org/abs/#1}{arXiv:#1}}

\newcommand{\Pic}{\operatorname{Pic}}

\newcommand{\Sing}{\operatorname{Sing}}

\newcommand{\Proj}{{\mathbb P}}

\newcommand{\Aut}{\operatorname{Aut}}

\def\geq{\geqslant}

\def\bibaut#1{{\sc #1}}

\def\phi{\varphi}
\def\ro[#1]{{\textcolor{red}{#1}}}

\begin{document}

\begin{abstract}
Let $\overline{\mathcal{M}}_{g,n}$ be the moduli stack parametrizing Deligne-Mumford stable $n$-pointed genus $g$ curves and let $\overline{M}_{g,n}$ be its coarse moduli space: the Deligne-Mumford compactification of the moduli space of $n$-pointed genus $g$ smooth curves. We prove that the automorphism groups of $\overline{\mathcal{M}}_{g,n}$ and $\overline{M}_{g,n}$ are isomorphic to the symmetric group on $n$ elements $S_{n}$ for any $g,n$ such that $2g-2+n\geq 3$, and compute the remaining cases.
\end{abstract}

\title{The Automorphism group of $\overline{M}_{g,n}$}

\author[Alex Massarenti]{Alex Massarenti}
\address{\sc Alex Massarenti\\
SISSA\\
via Bonomea 265\\
34136 Trieste\\ Italy}
\email{alex.massarenti@sissa.it}

\date{\today}


\maketitle
\tableofcontents

\section*{Introduction}
The search for an object parametrizing $n$-pointed genus $g$ smooth curves is a very classical problem in algebraic geometry. In \cite{DM} \textit{P. Deligne} and \textit{D. Mumford} proved that there exists an irreducible scheme $M_{g,n}$ coarsely representing the moduli functor of $n$-pointed genus $g$ smooth curves. Furthermore they provided a compactification $\overline{M}_{g,n}$ of $M_{g,n}$ adding Deligne-Mumford stable curves as boundary points and pointed out that the obstructions to representing the moduli functor of Deligne-Mumford stable curves in the category of schemes came from automorphisms of the curves. However this moduli functor can be represented in the category of algebraic stacks, indeed there exists a smooth Deligne-Mumford algebraic stack $\overline{\mathcal{M}}_{g,n}$ parametrizing Deligne-Mumford stable curves.
The stack $\overline{\mathcal{M}}_{g,n}$ and its coarse moduli space $\overline{M}_{g,n}$ from several decades are among the most studied objects in algebraic geometry, despite this many natural questions about their biregular and birational geometry remain unanswered. In particular we are interested in the following issue:
\begin{question}
What are the automorphism groups of $\overline{M}_{g,n}$ and $\overline{\mathcal{M}}_{g,n}$ ?
\end{question}
The biregular automorphism of the moduli space $M_{g,n}$ of $n$-pointed genus $g$-stable curves and of its Deligne-Mumford compactification $\overline{M}_{g,n}$ has been studied in a series of papers, for instance \cite{BM1} and \cite{Ro}.

Recently, in \cite{BM1} and \cite{BM2}, \textit{A. Bruno} and \textit{M. Mella} studied the fibrations of $\overline{M}_{0,n}$ using its description as the closure of the subscheme of the Hilbert scheme parametrizing rational normal curves passing through $n$ points in linearly general position in $\mathbb{P}^{n-2}$ given by \textit{M. Kapranov} in \cite{Ka}. It was expected that the only possible biregular automorphism of $\overline{M}_{0,n}$ were the ones associated to a permutation of the markings. Indeed \textit{Bruno} and \textit{Mella} as a consequence of their theorem on fibrations derive that the automorphism group of $\overline{M}_{0,n}$ is the symmetric group $S_{n}$ for any $n\geq 5$ \cite[Theorem 4.3]{BM2}.

The aim of this work is to extend \cite[Theorem 4.3]{BM2} to arbitrary values of $g,n$ and to the stack $\overline{\mathcal{M}}_{g,n}$. Our main result can be stated as follows.
\begin{theoremn}
Let $\overline{\mathcal{M}}_{g,n}$ be the moduli stack parametrizing Deligne-Mumford stable $n$-pointed genus $g$ curves, and let $\overline{M}_{g,n}$ be its coarse moduli space. If $2g-2+n\geq 3$ then 
$$\Aut(\overline{\mathcal{M}}_{g,n})\cong\Aut(\overline{M}_{g,n})\cong S_{n}$$
the symmetric group on $n$ elements. For $2g-2+n < 3$ we have the following special behavior: 
\begin{itemize}
\item[-] $\Aut(\overline{M}_{1,2})\cong (\mathbb{C}^{*})^{2}$ while $\Aut(\overline{\mathcal{M}}_{1,2})$ is trivial,
\item[-] $\Aut(\overline{M}_{0,4})\cong\Aut(\overline{\mathcal{M}}_{0,4})\cong\Aut(\overline{M}_{1,1})\cong PGL(2)$ while $\Aut(\overline{\mathcal{M}}_{1,1})\cong \mathbb{C}^{*}$,
\item[-] $\Aut(\overline{M}_{g})$ and $\Aut(\overline{\mathcal{M}}_{g})$ are trivial for any $g\geq 2$.  
\end{itemize}
\end{theoremn} 
These issues have been investigated in the Teichm\"uller-theoretic literature on the automorphism of moduli spaces $M_{g,n}$ developed in a series of papers by \textit{H.L. Royden}, \textit{C. J. Earle}, \textit{I. Kra}, \textit{M. Korkmaz}, and others, \cite{Ro}, \cite{EK} \cite{Ko}. A fundamental result, proved by \textit{Royden} in \cite{Ro}, states that the moduli space $M_{g,n}^{un}$ of genus $g$ smooth curve marked by $n$ unordered points has no non-trivial automorphism if $2g-2+n\geq 3$ which is exactly our bound.

Note that in the cases $g = n = 1$ and $g = 1, n = 2$ the automorphism group of the stack differs from that of the moduli space. This is particularly evident for $\overline{M}_{1,1}$, it is well known that $\overline{M}_{1,1}\cong \mathbb{P}^{1}$ and $\overline{\mathcal{M}}_{1,1}\cong \mathbb{P}(4,6)$. Clearly $\mathbb{P}^{1}\cong \mathbb{P}(4,6)$ as varieties, however they are not isomorphic as stacks, indeed $\mathbb{P}(4,6)$ has two stacky points with stabilizers $\mathbb{Z}_{4}$ and $\mathbb{Z}_{6}$. These two points are fixed by any automorphisms of $\mathbb{P}(4,6)$ while they are indistinguishable from any other point on the coarse moduli space $\overline{M}_{1,1}$.

The proof of the main Theorem is essentially divided into two parts: the cases $2g-2+n\geq 3$ and $2g-2+n < 3$.

When $2g-2+n\geq 3$ the main tool is \cite[Theorem 0.9]{GKM} in which \textit{A. Gibney, S. Keel} and \textit{I. Morrison} give an explicit description of the fibrations $\overline{M}_{g,n}\rightarrow X$ of $\overline{M}_{g,n}$ on a projective variety $X$ in the case $g\geq 1$. This result, combined with the triviality of the automorphism group of the generic curve of genus $g\geq 3$, let us to prove that the automorphism group of $\overline{M}_{g,1}$ is trivial for any $g\geq 3$. Since every genus $2$ curve is hyperelliptic and has a non trivial automorphism: the hyperelliptic involution, the argument used in the case $g\geq 3$ completely fails. So we adopt a different strategy: first we prove that any automorphism of $\overline{M}_{2,1}$ preserves the boundary and then we apply a famous theorem of \textit{H. L. Royden} \cite[Theorem 6.1]{Mok} to conclude that $\Aut(\overline{M}_{2,1})$ is trivial.

Then, applying \cite[Theorem 0.9]{GKM} we construct a morphism of groups between the group $\Aut(\overline{M}_{g,n})$ and $S_{n}$. Finally we generalize \textit{A. Bruno} and \textit{M. Mella}'s result proving that $\Aut(\overline{M}_{g,n})$ is indeed isomorphic to $S_{n}$ when $2g-2+n\geq 3$.

When $2g-2+n < 3$ a case by case analysis is needed. In particular the case $g = 1, n = 2$ requires an explicit description of the moduli space $\overline{M}_{1,2}$. Carefully analyzing the geometry of this surface we prove that $\overline{M}_{1,2}$ is isomorphic to a weighted blow up of $\mathbb{P}(1,2,3)$ in the point $[1:0:0]$, in particular $\overline{M}_{1,2}$ is toric. From this we derive that $\Aut(\overline{M}_{1,2})$ is isomorphic to $(\mathbb{C}^{*})^{2}$.

Finally we consider the moduli stack $\overline{\mathcal{M}}_{g,n}$. The canonical map $\overline{\mathcal{M}}_{g,n}\rightarrow\overline{M}_{g,n}$ induces a morphism or groups $\Aut(\overline{\mathcal{M}}_{g,n})\rightarrow\Aut(\overline{M}_{g,n})$. Since this morphism is injective as soon as the general $n$-pointed genus $g$ curve is automorphisms free we easily derive that the automorphism group of the stack $\overline{\mathcal{M}}_{g,n}$ is isomorphic to $S_{n}$ if $2g-2+n\geq 3$. Then we show that $\Aut(\overline{\mathcal{M}}_{1,2})$ is trivial using the fact that the canonical divisor of $\overline{\mathcal{M}}_{1,2}$ is a multiple of a boundary divisor.

This paper is organized as follows: in Section \ref{NP} we recall some basic facts about the moduli space $\overline{M}_{g,n}$ and the moduli stack $\overline{\mathcal{M}}_{g,n}$, furthermore we prove some preliminary results on the fibrations of $\overline{M}_{1,n}$, in Section \ref{sectionm12} we describe explicitly the moduli space $\overline{M}_{1,2}$, in Section \ref{a} we develop the case $2g-2+n\geq 3$, finally in Section \ref{stack} we study the automorphisms of the stack $\overline{\mathcal{M}}_{g,n}$.

\section{Notation and Preliminaries}\label{NP}
We work over the field of complex numbers. Let us recall some basic facts about the moduli space $\overline{M}_{g,n}$ parametrizing $n$-pointed stable curves of arithmetic genus $g$, and about the moduli stack $\overline{\mathcal{M}}_{g,n}$.

\subsubsection*{Nodal curves}
The arithmetic genus $g$ of a connected curve $C$ is defined as $g = h^{1}(C,\mathcal{O}_{C})$. Suppose that $C$ has at most nodal singularities. Let $C = \bigcup_{i=1}^{\gamma}C_{i}$ be the irreducible components decomposition of $C$, and set $\delta :=\sharp \Sing(C)$. Let
$$\nu:\overline{C} = \bigsqcup_{i=1}^{\gamma}\overline{C_{i}}\rightarrow C$$
be the normalization of $C$. The associated morphism $\mathcal{O}_{C}\hookrightarrow\mathcal{O}_{\overline{C}}$ on the structure sheaves yield the following sequence in cohomology
$$0\mapsto H^{0}(C,\mathcal{O}_{C})\rightarrow H^{0}(\overline{C},\mathcal{O}_{\overline{C}})\rightarrow \mathbb{C}^{\delta}\rightarrow H^{1}(C,\mathcal{O}_{C})\rightarrow H^{1}(\overline{C},\mathcal{O}_{\overline{C}})\mapsto 0.$$
We get a formula for the arithmetic genus $g$ of $C$
$$g = h^{1}(\overline{C},\mathcal{O}_{\overline{C}})+\delta-\gamma + 1 = \sum_{i=1}^{\gamma}g_{i} + \delta-\gamma + 1$$
where $g_{i} = h^{1}(\overline{C_{i}},\mathcal{O}_{\overline{C_{i}}})$ is the geometric genus of $C_{i}$.

\begin{Definition}
A \textit{stable $n$-pointed curve} is a complete connected curve $C$ that has at most nodal singularities, with an ordered collection $x_{1},\dots,x_{n}\in C$ of distinct smooth points of $C$, such that the $(n+1)$-tuple $(C,x_{1},\dots,x_{n})$ has finitely many automorphisms.
\end{Definition}

This finiteness condition is equivalent to say that every rational component of the normalization of $C$ has at least $3$ points lying over singular or marked points of $C$.\\
Moduli spaces of smooth algebraic curves have been defined and then compactified adding stable curves by \textit{Deligne} and \textit{Mumford} in \cite{DM}. Furthermore \textit{Deligne} and \textit{Mumford} proved that, if $2g-2+n>0$, there exists a coarse moduli space $\overline{M}_{g,n}$ parametrizing isomorphism classes of $n$-pointed stable curves of arithmetic genus $g$, and this space is an irreducible projective variety of dimension $3g-3+n$. 

\subsubsection*{Boundary of $\overline{M}_{g,n}$ and dual modular graphs}
The points in the boundary $\partial\overline{M}_{g,n}$ of the moduli space $\overline{M}_{g,n}$ represent isomorphisms classes of singular pointed stable curves. The geometry of such curves is encoded in a graph, called dual modular graph. The boundary has a stratification whose loci, called strata, parametrize curves of a certain topological type and with a fixed configuration of the marked points.\\
Each nodal curve has an associated graph. This allows to represent nodal curves in a very simple way and translate some issues related to nodal curves in the language of graph theory.\\
Let $C$ be a connected nodal curve with $\gamma$ irreducible components and $\delta$ nodes. The dual graph $\Gamma_{C}$ of $C$ is the graph whose vertexes represent the irreducible components of $C$ and whose edges represent nodes lying on two components.\\ 
More precisely, each irreducible component is represented by a vertex labeled by two numbers: the genus and the number of marked points of the component. An edge connecting two vertex means that the two corresponding components intersect in the node corresponding to the edge. A loop on a vertex means that the corresponding component has a self-intersection.\\
Recently, \textit{S. Maggiolo} and \textit{N. Pagani} developed a software that generates all stable dual graphs for prescribed values of $g,n$ whose detailed description can be found in \cite{MP}. We will use this package to generate graphs needed in this paper.\\
We denote by $\Delta_{irr}$ the locus in $\overline{M}_{g,n}$ parametrizing irreducible nodal curves with $n$ marked points, and by $\Delta_{i,P}$ the locus of curves with a node which divides the curve into a component of genus $i$ containing the points indexed by $P$ and a component of genus $g-i$ containing the remaining points.\\
The closures of the loci $\Delta_{irr}$ and $\Delta_{i,P}$ are the irreducible components of the boundary $\partial\overline{M}_{g,n}$, see \cite[Proposition 1.21]{Mor}. 

\subsubsection*{Forgetful morphisms}
For any $i=1,\dots,n$ there is a canonical forgetful morphism 
$$\pi_{i}:\overline{M}_{g,n}\rightarrow\overline{M}_{g,n-1}$$ 
forgetting the $i$-th marked point. If $g > 2$ and $[C,x_{1},\dots,\hat{x_{i}},\dots,x_{n}]\in\overline{M}_{g,n-1}$ is a general point the fiber 
$$\pi_{i}^{-1}([C,x_{1},\dots,\hat{x_{i}},\dots,x_{n}])\cong C$$ 
is isomorphic to $C$ and $\pi_{i}$ plays the role of the universal curve.\\ 
Note that if $n\geq 2$ the fiber $\pi_{i}^{-1}([C,x_{1},\dots,\hat{x_{i}},\dots,x_{n}])$ always intersects the boundary of $\overline{M}_{g,n}$, in fact the points of the fiber corresponding to marked points represent singular curves with two irreducible components: $C$ itself and  a $\mathbb{P}^{1}$ with two marked points and intersecting $C$ in a point.\\
In the same way for any $I\subseteq \{1,\dots,n\}$ we have a forgetful map $\pi_{I}:\overline{M}_{g,n}\rightarrow\overline{M}_{g,n-|I|}$. The map $\pi_{i}$ has sections $s_{i,j}:\overline{M}_{g,n-1}\rightarrow\overline{M}_{g,n}$ defined by sending the point $[C,x_{1},\dots,\hat{x_{i}},\dots,x_{n}]$ to the isomorphism class of the $n$-pointed genus $g$ curve obtained by attaching at $x_{j}\in C$ a $\mathbb{P}^{1}$ with two marked points labeled by $x_{i}$ and $x_{j}$.

\subsubsection*{The universal curve}
The moduli space $\overline{M}_{g,1}$ with the forgetful morphism \mbox{$\pi:\overline{M}_{g,1}\rightarrow\overline{M}_{g}$} at first glance seems to play the role of  the universal curve over $\overline{M}_{g}$.\\
However,  on closer examination one realizes that $\pi^{-1}([C])\cong C$ if and only if $[C]\in\overline{M}_{g}^{0}$ the locus of automorphisms-free curves.\\ 
It is well known that the set-theoretic fiber of $\pi:\overline{M}_{g,1}\rightarrow\overline{M}_{g}$ over $[C]\in\overline{M}_{g}$ is the quotient $C/\Aut(C)$.\\ 
For example over an open subset of $\overline{M}_{2}$ the fibration $\pi:\overline{M}_{2,1}\rightarrow\overline{M}_{2}$ is a $\mathbb{P}^{1}$-bundle and this is true even scheme-theoretically.

\begin{Remark}
The situation is different if instead  of considering the moduli space $\overline{M}_{g,1}$ we consider the Deligne-Mumford moduli stack $\overline{\mathcal{M}}_{g,1}$.\\ 
In fact,  in this case the  fiber $\pi^{-1}([C])$ is isomorphic to $C$ and via the morphism $\pi:\overline{\mathcal{M}}_{g,1}\rightarrow\overline{\mathcal{M}}_{g}$ the stack $\overline{\mathcal{M}}_{g,1}$ plays the role of the universal curve over $\overline{\mathcal{M}}_{g}$.
\end{Remark}

\subsubsection*{Divisor classes on $\overline{\mathcal{M}}_{g,n}$}
Let us briefly recall the definitions of classes $\lambda$ and $\psi_{i}$ on $\overline{\mathcal{M}}_{g,n}$. Consider the forgetful morphism $\pi:\overline{\mathcal{M}}_{g,n+1}\rightarrow\overline{\mathcal{M}}_{g,n}$ forgetting one of the marked points and its sections $\sigma_{1},\dots,\sigma_{n}:\overline{\mathcal{M}}_{g,n}\rightarrow\overline{\mathcal{M}}_{g,n+1}$.\\ 
Let $\omega_{\pi}$ be the relative dualizing sheaf of the morphism $\pi$. The Hodge class is defined as
$$\lambda := c_{1}(\pi_{*}(\omega_{\pi})).$$
The classes $\psi_{i}$ are defined as
$$\psi_{i} := \sigma_{i}^{*}(c_{1}(\omega_{\pi}))$$
for any $i = 1,\dots,n$.\\ 
Finally we denote by $\delta_{irr}$ and $\delta_{i,P}$ the boundary classes on $\overline{\mathcal{M}}_{g,n}$.

\subsubsection*{Cyclic quotient singularities}
Any cyclic quotient singularity is of the form $\mathbb{A}^{n}/\mu_{r}$, where $\mu_{r}$ is the group of $r$-roots of unit. The action $\mu_{r}\curvearrowright\mathbb{A}^{n}$ can be diagonalized, and then written in the form 
$$\mu_{r}\times\mathbb{A}^{n}\rightarrow\mathbb{A}^{n}, \: (\epsilon,x_{1},\dots,x_{n})\mapsto (\epsilon^{a_{1}}x_{1},\dots,\epsilon^{a_{n}}x_{n}),$$
for some $a_{1},\dots,a_{r}\in\mathbb{Z}/\mathbb{Z}_{r}$. The singularity is thus determined by the numbers $r,a_{1},\dots,a_{n}$.\\ 
Following the notation set by \textit{M. Reid} in \cite{Re}, we denote by $\frac{1}{r}(a_{1},\dots,a_{n})$ this type of singularity.

\subsubsection*{Fibrations of $\overline{M}_{g,n}$}
The following result by \textit{A. Gibney, S. Keel} and \textit{I. Morrison} gives an explicit description of the fibrations
$\overline{M}_{g,n}\rightarrow X$ of $\overline{M}_{g,n}$ on a projective variety $X$ in the case $g\geq 1$. We denote by $N$ the set $\{1,\dots,n\}$ of the markings, if $S\subset N$ then $S^{c}$ denotes its complement. 

\begin{Theorem}(\underline{\textit{Gibney - Keel - Morrison}})\label{GKM}
Let $D\in \Pic(\overline{M}_{g,n})$ be a nef divisor. 
\begin{itemize}
\item[-] If $g\geq 2$ either $D$ is the pull-back of a nef divisor on $\overline{M}_{g,n-1}$ via one of the forgetful morphisms or $D$ is big and the exceptional locus of $D$ is contained in $\partial\overline{M}_{g,n}$.
\item[-] If $g = 1$ either $D$ is the tensor product of pull-backs of nef divisors on $\overline{M}_{1,S}$ and $\overline{M}_{1,S^{c}}$ via the tautological projection for some subset $S\subseteq N$ or $D$ is big and the exceptional locus of $D$ is contained in $\partial\overline{M}_{g,n}$.
\end{itemize}
\end{Theorem} 

The above theorem will be crucial to determine the automorphism group of $\overline{M}_{g,n}$, and can be found in \cite[Theorem 0.9]{GKM}.\\ 
An immediate consequence of \ref{GKM} is that for $g \geq 2$ any fibration of $\overline{M}_{g,n}$ to a projective variety factors through a projection to some $\overline{M}_{g,i}$ with $i < n$, while $\overline{M}_{g}$ has no non-trivial fibrations. This last fact had already been shown by \textit{A. Gibney} in her Ph.D. Thesis \cite{G}.\\
Such a clear description of the fibrations of $\overline{M}_{g,n}$ is no longer true for $g = 1$, an explicit counterexample to this fact was given by \textit{R. Pandharipande} and can be found in \cite[Example A.2]{BM2}, see also \cite{Pa} for similar constructions. However, if we consider the fibrations of the type
  \[
  \begin{tikzpicture}[xscale=2.3,yscale=-1.2]
    \node (A0_0) at (0, 0) {$\overline{M}_{1,n}$};
    \node (A0_1) at (1, 0) {$\overline{M}_{1,n}$};
    \node (A0_2) at (2, 0) {$\overline{M}_{1,n-1}$};
    \path (A0_0) edge [->]node [auto] {$\scriptstyle{\phi}$} (A0_1);
    \path (A0_1) edge [->]node [auto] {$\scriptstyle{\pi_{i}}$} (A0_2);
  \end{tikzpicture}
  \]
where $\phi$ is an automorphism of $\overline{M}_{1,n}$, thanks to the second part of Theorem \ref{GKM} we can prove the following lemma.

\begin{Lemma}\label{g1}
Let $\phi$ be an automorphism of $\overline{M}_{1,n}$. Any fibration of the type $\pi_{i}\circ\phi$ factorizes through a forgetful morphism $\pi_{j}:\overline{M}_{1,n}\rightarrow\overline{M}_{1,n-1}$.
\end{Lemma}
\begin{proof}
By the second part of Theorem \ref{GKM} the fibration $\pi_{i}\circ\phi$ factorizes through a product of forgetful morphisms $\pi_{S^{c}}\times\pi_{S}:\overline{M}_{1,n}\rightarrow\overline{M}_{1,S}\times_{\overline{M}_{1,1}}\overline{M}_{1,S^{c}}$ and we have a commutative diagram
  \[
  \begin{tikzpicture}[xscale=2.9,yscale=-1.2]
    \node (A0_0) at (0, 0) {$\overline{M}_{1,n}$};
    \node (A0_1) at (1, 0) {$\overline{M}_{1,n}$};
    \node (A1_0) at (0, 1) {$\overline{M}_{1,S}\times_{\overline{M}_{1,1}}\overline{M}_{1,S^{c}}$};
    \node (A1_1) at (1, 1) {$\overline{M}_{1,n-1}$};
    \path (A0_0) edge [->]node [auto] {$\scriptstyle{\phi}$} (A0_1);
    \path (A1_0) edge [->]node [auto] {$\scriptstyle{\overline{\phi}}$} (A1_1);
    \path (A0_1) edge [->]node [auto] {$\scriptstyle{\pi_{i}}$} (A1_1);
    \path (A0_0) edge [->]node [auto,swap] {$\scriptstyle{\pi_{S^{c}}\times\pi_{S}}$} (A1_0);
  \end{tikzpicture}
  \]
The fibers of $\pi_{i}$ and $\pi_{S^{c}}\times\pi_{S}$ are both $1$-dimensional. Furthermore $\phi$ maps the fiber of $\pi_{S^{c}}\times\pi_{S}$ over $([C,x_{a_{1}},\dots,x_{a_{s}}],[C,x_{b_{1}},\dots,x_{b_{n-s}}])$ to $\pi_{i}^{-1}(\overline{\phi}([C,x_{a_{1}},\dots,x_{a_{s}}],[C,x_{b_{1}},\dots,x_{b_{n-s}}]))$.\\
Take a point $[C,x_{1},\dots,x_{n-1}]\in\overline{M}_{1,n-1}$, the fiber $\pi_{i}^{-1}([C,x_{1},\dots,x_{n-1}])$ is mapped isomorphically to a fiber $\Gamma$ of $\pi_{S^{c}}\times\pi_{S}$ which is contracted to a point $y = (\pi_{S^{c}}\times\pi_{S})(\Gamma)$. The map 
$$\overline{\psi}:\overline{M}_{1,n-1}\rightarrow\overline{M}_{1,S}\times_{\overline{M}_{1,1}}\overline{M}_{1,S^{c}}, \; [C,x_{1},\dots,x_{n-1}]\mapsto y,$$
is clearly the inverse of $\overline{\phi}$. So $\overline{\phi}$ defines a bijective morphism between $\overline{M}_{1,S}\times_{\overline{M}_{1,1}}\overline{M}_{1,S^{c}}$ and $\overline{M}_{1,n-1}$, and since $\overline{M}_{1,n-1}$ is normal $\overline{\phi}$ is an isomorphism.\\
This forces $S = \{j\}$, $S^{c} = \{1,\dots,\overline{j},\dots,n\}$. So we reduce to the commutative diagram
\[
  \begin{tikzpicture}[xscale=2.9,yscale=-1.2]
    \node (A0_0) at (0, 0) {$\overline{M}_{1,n}$};
    \node (A0_1) at (1, 0) {$\overline{M}_{1,n}$};
    \node (A1_0) at (0, 1) {$\overline{M}_{1,1}\times_{\overline{M}_{1,1}}\overline{M}_{1,n-1}$};
    \node (A1_1) at (1, 1) {$\overline{M}_{1,n-1}$};
    \path (A0_0) edge [->]node [auto] {$\scriptstyle{\phi}$} (A0_1);
    \path (A1_0) edge [->]node [auto] {$\scriptstyle{\overline{\phi}}$} (A1_1);
    \path (A0_1) edge [->]node [auto] {$\scriptstyle{\pi_{i}}$} (A1_1);
    \path (A0_0) edge [->]node [auto,swap] {$\scriptstyle{\pi_{S^{c}}\times\pi_{j}}$} (A1_0);
  \end{tikzpicture}
  \]
and $\pi_{i}\circ\phi$ factorizes through the forgetful morphism $\pi_{j}$.
\end{proof}

\section{The moduli space of $2$-pointed elliptic curves}\label{sectionm12}

Let $(C,p)$ be a nodal elliptic curve. Then there exists $(a,b)\in \mathbb{A}^{2}\setminus (0,0)$ such that $(C,p)$ is isomorphic to $(C^{'},[0:1:0])$, where
$$C^{'} = Z(zy^{2}-x^{3}-axz^{2}-bz^{3})\subset\mathbb{P}^{2}.$$
This representation is called \textit{Weierstrass representation} of the elliptic curve. Consider now the $4$-fold
$$X := Z(zy^{2}-x^{3}-axz^{2}-bz^{3})\subset\mathbb{A}^{3}_{0}\times \mathbb{A}^{2}_{0}.$$
There is an action of $\mathbb{C}^{*}\times\mathbb{C}^{*}\curvearrowright X$ given by 
$$\mathbb{C}^{*}\times\mathbb{C}^{*}\times X\rightarrow X, \: ((\lambda,\xi),(x,y,z,a,b))\mapsto (\xi\lambda^{2}x,\xi\lambda^{3}y,\xi z,\lambda^{4}a,\lambda^{6}b).$$
The moduli stack $\overline{\mathcal{M}}_ {1,1}$ is the quotient stack $[\mathbb{A}^{2}\setminus (0,0)/\mathbb{C}^{*}]\cong \mathbb{P}(4,6)$ and the moduli space $\overline{M}_{1,1}$ is the quotient $\mathbb{A}^{2}\setminus (0,0)/\mathbb{C}^{*}\cong\mathbb{P}^{1}$.\\ 
There are two points of $\overline{\mathcal{M}}_{1,1}$ that are stabilized by the action of $\mu_{4}$ and $\mu_{6}$ respectively. These are classes of curves whose Weierstrass representations can be chosen respectively as:
$$C_{4} := \{y^{2}z = x^{3}+xz^{2}\}\subset\mathbb{P}^{2},$$ 
$$C_{6} := \{y^{2}z = x^{3}+z^{3}\}\subset\mathbb{P}^{2}.$$
Now, $\overline{\mathcal{M}}_{1,2}$ is the universal curve over $\overline{\mathcal{M}}_{1,1}$, so $\overline{\mathcal{M}}_{1,2} = [X/\mathbb{C}^{*}\times\mathbb{C}^{*}]$ and $\overline{M}_{1,2} = X/\mathbb{C}^{*}\times\mathbb{C}^{*}$.\\ 
In order to determine the singularities of $\overline{M}_{1,2}$ we have to analyze carefully the action $\mathbb{C}^{*}\times\mathbb{C}^{*}\curvearrowright X$.\\
Since $\overline{\mathcal{M}}_{1,2}$ is a smooth Deligne-Mumford stack the coarse moduli space $\overline{M}_{1,2}$ will have finite quotient singularities at the places where the automorphism groups jump. Let $(C,p)$ be a elliptic curve over $\mathbb{C}$, it is well known that
\begin{itemize}
\item[-] $|\Aut(C,p)| = 2$ if $j(C)\neq 0,1728$,
\item[-] $|\Aut(C,p)| = 4$ if $j(C) = 1728$,
\item[-] $|\Aut(C,p)| = 6$ if $j(C)\neq 0$.
\end{itemize}
Adding a marked point will kill some automorphisms. We expect that points of type $(C,p,q)$ with $|\Aut(C,p)| = 2$ will have trivial automorphism group. Automorphisms will jump on the points $(C,p,q)$ with $|\Aut(C,p)| = 4,6$.\\ 
To understand the behavior of the boundary $\partial\overline{M}_{1,2}$ we have to observe the following possible degenerations.
\begin{itemize}
\item[-] The divisor $\Delta_{irr}$ whose general point is a curve with dual graph
$$
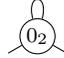
\begin{tikzpicture}[scale=.5,inner sep=10][baseline]
      \path(0,0) ellipse (2 and 1);
      \tikzstyle{level 1}=[counterclockwise from=-150,level distance=9mm,sibling angle=120]
      \ndo (A0) at (0:1) {$\scriptstyle{0_{2}}$} child child;

      \draw (A0) .. controls +(75.000000:1.2) and +(105.000000:1.2) .. (A0);
    \end{tikzpicture}
$$
and so automorphisms free.
\item[-] The divisor $\Delta_{0,2}$ whose general point is a curve with dual graph
$$
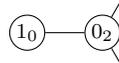
\begin{tikzpicture}[scale=.5,inner sep=10][baseline]
      \path(0,0) ellipse (2 and 1);
      \tikzstyle{level 1}=[counterclockwise from=-60,level distance=9mm,sibling angle=120]
      \ndo (A0) at (0:1) {$\scriptstyle{0_{2}}$} child child;
      \ndo (A1) at (180:1) {$\scriptstyle{1_{0}}$};

      \path (A0) edge [bend left=0.000000] (A1);
\end{tikzpicture}
$$
and so with two automorphisms coming from the elliptic involution. Here we expect to get two singular points when the number of automorphisms of the elliptic curve jumps to $4$ and $6$.
\item[-] Two further degenerations in codimension two with the following dual graphs.
$$
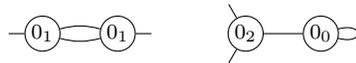
\begin{tabular}{ccc}
\begin{tikzpicture}[scale=.5,inner sep=10][baseline]
      \path(0,0) ellipse (2 and 1);
      \tikzstyle{level 1}=[counterclockwise from=0,level distance=9mm,sibling angle=0]
      \ndo (A0) at (0:1) {$\scriptstyle{0_{1}}$} child;
      \tikzstyle{level 1}=[counterclockwise from=180,level distance=9mm,sibling angle=0]
      \ndo (A1) at (180:1) {$\scriptstyle{0_{1}}$} child;
	  \path (A0) edge [bend left=-15.000000] (A1);
      \path (A0) edge [bend left=15.000000] (A1);
\end{tikzpicture}
&
\qquad
&
\begin{tikzpicture}[scale=.5,inner sep=10][baseline]
      \path(0,0) ellipse (2 and 1);
      \ndo (A0) at (0:1) {$\scriptstyle{0_{0}}$};
      \tikzstyle{level 1}=[counterclockwise from=120,level distance=9mm,sibling angle=120]
      \ndo (A1) at (180:1) {$\scriptstyle{0_{2}}$} child child;
	  \draw (A0) .. controls +(-15.000000:1.2) and +(15.000000:1.2) .. (A0);
      \path (A0) edge [bend left=0.000000] (A1);
\end{tikzpicture}
\end{tabular}
$$
Here the automorphism group remains of order two, so we do not expect to have singularities.
\end{itemize}

\begin{Proposition}\label{M12}
The moduli space $\overline{M}_{1,2}$ is a rational surface with four singular points. Two singular points lie in $M_{1,2}$, and are:
\begin{itemize}
\item[-] a singularity of type $\frac{1}{4}(2,3)$ representing an elliptic curve of Weierstrass representation $C_{4}$ with marked points $[0:1:0]$ and $[0:0:1]$;
\item[-] a singularity of type $\frac{1}{3}(2,4)$ representing an elliptic curve of Weierstrass representation $C_{6}$ with marked points $[0:1:0]$ and $[0:1:1]$.
\end{itemize}
The remaining two singular points lie on the boundary divisor $\Delta_{0,2}$, and are:
\begin{itemize}
\item[-] a singularity of type $\frac{1}{6}(2,4)$ representing a reducible curve whose irreducible components are an elliptic curve of type $C_{6}$ and a smooth rational curve connected by a node;
\item[-] a singularity of type $\frac{1}{4}(2,6)$ representing a reducible curve whose irreducible components are an elliptic curve of type $C_{4}$ and a smooth rational curve connected by a node.
\end{itemize}
\end{Proposition}
\begin{proof}
The rationality of $\overline{M}_{1,2}$ follows from the fact that the forgetful map $\overline{M}_{1,2}\rightarrow\overline{M}_{1,1}$ realizes $\overline{M}_{1,2}$ as a ruled surface over $\mathbb{P}^{1}$.\\ 
To compute the singularities we study the action on $X$. Note that on $X$, $z=0 \Rightarrow x=0 \Rightarrow y\neq 0$. So $X$ is covered by the charts $\{z\neq 0\}$ and $\{y\neq 0\}$.\\
Consider first the chart $\{z\neq 0\}$. On this chart $X$ is given by $\{y^{2} = x^{3}+ax+b\}$ so $b = y^{2}-x^{3}-ax$. We can take $(x,y,a)$ as coordinates, and the action of $\mathbb{C}^{*}\times\mathbb{C}^{*}$ is given by $(\lambda,x,y,a)\mapsto (\lambda^{2}x,\lambda^{3}y,\lambda^{4}a)$. The point $(0,0,0)$ is stabilized by $\mathbb{C}^{*}\times\mathbb{C}^{*}$, so does not produce any singularity. Since $(2,3) = (3,4) = 1$ the points $(x,y,a)$ such that $xy\neq 0$ or $ya\neq 0$ have trivial stabilizer.\\
If $y = 0$ the action is given by $(\lambda,x,a)\mapsto (\lambda^{2}x,\lambda^{4}a)$. We distinguish two cases.
\begin{itemize}
\item[-] If $x = 0$ then $a\neq 0$, the stabilizer is $\mu_{4}$. So on the chart $a\neq 0$ we have a singularity of type $\frac{1}{4}(2,3)$. Note that $x=y=0$ implies $b = 0$. The singular point corresponds to a smooth elliptic curve of Weierstrass form $C_{4}$ and whose second marked point is $[0:0:1]$.
\item[-] If $x\neq 0$ then the stabilizer is $\mu_{2}$ and on this chart we find points of type $\frac{1}{2}(1,0)$ and these are smooth points.
\end{itemize}
If $y\neq 0$, then $\lambda^{3} = 1$ and we get a singularity of type $\frac{1}{3}(2,4)$, that is a $A_{2}$ singularity, in the point $a=x=0$. This is a curve of type $C_{6}$ where we mark the point $[0:1:1]$. In $\overline{M}_{1,2}$ the singular point we found represents a smooth elliptic curve of Weierstrass form $C_{6}$ and whose second marked point is $[0:1:1]$.\\
Consider now the locus $\{z=0\}$. We can take $y = 1$ and $X$ is given by $\{z=x^{3}+axz^{2}+bz^{3}\}$. We are interested in a neighborhood of $x = z = 0$. Let $f(x,z,a,b) = z-x^{3}-axz^{2}-bz^{3}$ be the polynomial defining $X$. Since $\frac{\partial f}{\partial z}_{|z=0} \neq 0$ we can chose $(x,a,b)$ as local coordinates.\\ 
The action is given by $(\lambda,x,a,b)\mapsto (\lambda^{2}x,\lambda^{4}a,\lambda^{6}b)$. If $x\neq 0$ the stabilizer is trivial. If $x = 0$ and $ab\neq 0$ the stabilizer is $\mu_{2}$ and does not produce any singularity. We get the following two singular points.
\begin{itemize}
\item[-] If $a = 0, b\neq 0$ then we have a singular point of type $\frac{1}{6}(2,4)$. In this case we get an elliptic curve of type $C_{6}$ where we are taking the second marked point equal to the first $[0:1:0]$. So this singular point is a point on the boundary divisor $\Delta_{0,2}$ representing a reducible curve whose irreducible components are an elliptic curve of type $C_{6}$ and a smooth rational curve connected by a node. 
\item[-] If $a\neq 0, b = 0$ we get a singular point of type $\frac{1}{4}(2,6)$. We have an elliptic curve of type $C_{4}$ where the second marked point coincides with the first $[0:1:0]$. This singular point is a point on the boundary divisor $\Delta_{0,2}$ representing a reducible curve whose irreducible components are an elliptic curve of type $C_{4}$ and a smooth rational curve connected by a node.
\end{itemize}
These two points are the only singularities on the divisor $\Delta_{0,2}$.
\end{proof}

The rational Picard group of $\overline{M}_{1,2}$ is freely generated by the two boundary divisors \cite[Theorem 3.1.1]{Be}. The divisors $\Delta_{irr}$ and $\Delta_{0,2}$ are both smooth, rational curves. The boundary divisor $\Delta_{irr}$ has zero self intersection while $\Delta_{0,2}$ has negative self intersection.\\ In \cite{Sm} \textit{D.I. Smyth} proves that on $\overline{M}_{1,2}$ there exists a birational morphisms contracting $\Delta_{0,2}$. In the following we give a precise description of this contraction. Let us briefly recall the structure of a weighted blow up.

\begin{Remark}\label{cwbu}
Let $\pi_{\omega}:Y\rightarrow\mathbb{C}^{2}$ be the weighted blow up of $\mathbb{C}^{2}$ at the origin with weight $\omega = (\omega_{1},\omega_{2})$,
$$Y = \{((x,y),[u:v])\in \mathbb{C}^{2}\times \mathbb{P}(\omega_{1},\omega_{2}) \; | \; (x,y)\in \overline{[u:v]}\}.$$
Then $Y$ is given by the equation $x^{\omega_{1}}v-y^{\omega_{2}}u$ in $\mathbb{C}^{2}\times \mathbb{P}(\omega_{1},\omega_{2})$. The blow up surface $Y$ is covered by two chart.
\begin{itemize}
\item[-] On the chart $v = 1$ we have $x^{\omega_{1}} = y^{\omega_{2}}u$ and $\lambda^{\omega_{2}} = 1$. The action of $\mathbb{C}^{*}$ is given by $\lambda\cdot(y,u) = (\lambda^{\omega_{2}}y,\lambda^{\omega_{1}}u)$, so the point $x=y=u=0$ is a cyclic quotient singularity of type $\frac{1}{\omega_{2}}(\omega_{1},\omega_{2})$. 
\item[-] On the chart $u = 1$ we have $y^{\omega_{2}} = x^{\omega_{1}}v$ and $\lambda^{\omega_{1}} = 1$. The action of $\mathbb{C}^{*}$ is given by $\lambda\cdot(x,v) = (\lambda^{\omega_{1}}x,\lambda^{\omega_{2}}v)$, so the point $x=y=v=0$ is a cyclic quotient singularity of type $\frac{1}{\omega_{1}}(\omega_{1},\omega_{2})$. 
\end{itemize}
The singular points of $Y$ are cyclic quotient singularities located at the exceptional divisor. Actually they coincide with the origins of the two charts.
\end{Remark}

\begin{Theorem}\label{wbu}
The moduli space $\overline{M}_{1,2}$ is isomorphic to a weighted blow up of the weighted projective plane $\mathbb{P}(1,2,3)$ in its smooth point $[1:0:0]$. In particular $\overline{M}_{1,2}$ is a toric variety.
\end{Theorem}   
\begin{proof}
Recall the description of $\overline{M}_{1,2}$ given at the beginning of this section. On the chart $\mathcal{U}_{z} := \{z\neq 0\}$ we define a morphism 
$$f_{\mathcal{U}_{z}}:\mathcal{U}_{z}\rightarrow \mathbb{P}(1,2,3), \; (x,y,z,a,b)\mapsto (x,az^{2},bz^{3}).$$
Note that the action of $\mathbb{C}^{*}\times\mathbb{C}^{*}$ on this triple is given by $(\xi\lambda^{2},\xi^{2}\lambda^{4},\xi^{3}\lambda^{6})$, and $f_{\mathcal{U}_{z}}$ is indeed a well defined morphism to $\mathbb{P}(1,2,3)$.\\
On the open set $\{z\neq 0\}$ we can set $z = 1$ and ignore the action of $\xi$. If we forget $y$ we can derive it up to a sign and this corresponds to the action of $\lambda = -1$.\\
Note that the morphism $f_{\mathcal{U}_{z}}$ maps the two singular point in $M_{1,2}$ we found in Proposition \ref{M12} in the points $[0:1:0],[0:0:1]\in\mathbb{P}(1,2,3)$, which are the only singularities of the weighted projective plane and of the same type of the singularities on $M_{1,2}$.\\
On $\mathcal{U}_{y}:=\{y \neq 0\}$ the equation of $\overline{M}_{1,2}$ is $z=x^{3}+axz^{2}+bz^{3}$. So, as explained in the proof of Proposition \ref{M12} $x$ is a local parameter near $z=0$. We can consider the morphism 
$$f_{\mathcal{U}_{y}}(x,y,z,a,b) = \left( 1,a\left(\frac{x^{2}+az^{2}}{1-bz^{2}}\right)^{2}, b\left(\frac{x^{2}+az^{2}}{1-bz^{2}}\right)^{3}\right).$$
From this formulation it is clear that $f_{\mathcal{U}_{y}}$ is defined even on the locus $\{x = 0\}$ and the divisor $\Delta_{0,2} = \{x=z=0\}$ is contracted in the smooth point $[1:0:0]$ of $\mathbb{P}(1,2,3)$.\\
On $\mathcal{U}_{z}\cap\mathcal{U}_{y}$ we have $\frac{z}{x} = \frac{x^{2}+az^{2}}{1-bz^{2}}$ and $f_{\mathcal{U}_{z}} = f_{\mathcal{U}_{y}}$, so $f_{\mathcal{U}_{z}}, f_{\mathcal{U}_{y}}$ glue to a morphism 
$$f:\overline{M}_{1,2}\rightarrow \mathbb{P}(1,2,3).$$
Then $f$ is a blow up of $\mathbb{P}(1,2,3)$ in $[1:0:0]$ and $\Delta_{0,2}$ is the corresponding exceptional divisor. By Proposition \ref{M12} there are two singular points of type $\frac{1}{6}(2,4),\frac{1}{4}(2,6)$ on $\Delta_{0,2}$, and by Remark \ref{cwbu} the only way to obtain these two singularities is to perform a weighted blow up in $[1:0:0]$. 
\end{proof}

\begin{Remark}
The weighted projective space $\mathbb{P}(a_{0},\dots,a_{n})$ is defined by
$$\mathbb{P}(a_{0},\dots,a_{n}) = \Proj(S),$$
where $a_{0},\dots,a_{n}$ are positive integers and $S$ is the polynomial ring $k[x_{0},\dots,x_{n}]$, graded by $\deg(x_{i}) = a_{i}$.\\
Consider the set of vectors $V = \{e_{1},\dots,e_{n},e_{0} = -e_{1}-\dots -e_{n}\}$ in $\mathbb{R}^{n}$ and the fan whose cones are generated by proper subset of $V$ in the lattice generated by $\frac{1}{a_{1}}e_{i}$ for $i = 0,\dots,n$. The toric variety associated to this fan is $\mathbb{P}(a_{0},\dots,a_{n})$. For what follows it is particularly interesting the fan of $\mathbb{P}(1,2,3)$:
  
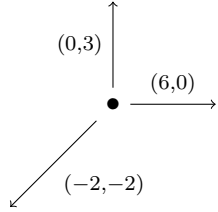
\[
  \begin{tikzpicture}[xscale=2.3,yscale=-1.5]
    \node (A0_1) at (1, 0) {};
    \node (A1_1) at (1, 1) {$\bullet$};
    \node (A1_2) at (2, 1) {};
    \node (A2_0) at (0, 2) {};
    \path (A1_1) edge [->]node [auto] {$\scriptstyle{(-2,-2)}$} (A2_0);
    \path (A1_1) edge [->]node [auto] {$\scriptstyle{(0,3)}$} (A0_1);
    \path (A1_1) edge [->]node [auto] {$\scriptstyle{(6,0)}$} (A1_2);
  \end{tikzpicture}
  \]
Note that $(6,0)+(0,3) = 2(3,1)$ and $(6,0)+(-2,-2) = 2(2,-1)$. These points correspond to the two singular points of $\mathbb{P}(1,2,3)$.
For a detailed toric description of the weighted projective space see \cite[Section 3]{Ji}.
\end{Remark}

\section{Automorphisms of $\overline{M}_{g,n}$}\label{a}

Our aim is to proceed by induction on $n$. The first step of induction is Proposition \ref{Mg1}.
In our argument the key fact is that the generic curve of genus $g > 2$ is automorphisms free. This is no longer true if $g = 2$ since every genus $2$ curve is hyperelliptic and has a non trivial automorphism: the hyperelliptic involution. So we adopt a different strategy. First we prove that any automorphism of $\overline{M}_{2,1}$ preserves the boundary and then we apply a famous theorem of \textit{H. L. Royden} which implies that $M_{g,n}^{un}$ (the moduli space of smooth genus $g$ curves with unordered marked points) admits no non-trivial automorphisms or unramified correspondences for $2g-2+n\geq 3$, see \cite[Theorem 6.1]{Mok}. In the case $g = 1$ the following observations will be crucial.

\begin{Remark}\label{un}
Let $[C,x_{1},x_{2}]$ be a two pointed elliptic curve and let $x_{1}$ be the origin of the group law on $C$. Let $\tau:C\rightarrow C$ be the translation mapping $x_{2}$ in $x_{1}$, and let $\eta$ be the elliptic involution. Then $\eta\circ\tau:C\rightarrow C$ is an automorphism of $C$ switching $x_{1}$ and $x_{2}$. Then $[C,x_{1},x_{2}] = [C,x_{2},x_{1}]$ and $\overline{M}_{1,2}\cong\overline{M}_{1,2}^{un}$.
\end{Remark}

\begin{Lemma}\label{delta02}
Any automorphism of $\overline{M}_{1,2}$ and $\overline{M}_{1,3}$ preserves the divisor $\Delta_{0,2}$.
\end{Lemma}
\begin{proof}
By Theorem \ref{wbu} the divisor $\Delta_{0,2}\subset\overline{M}_{1,2}$ is the only contractible, smooth, rational curve in $\overline{M}_{1,2}$. Then it is stabilized by any automorphism.\\
Let $\phi$ be an automorphism of $\overline{M}_{1,3}$ such that $\phi(\Delta_{0,2})\nsubseteq\Delta_{0,2}$ then composing $\phi$ with the morphism forgetting the marked point on the elliptic tail and considering the associated commutative diagram 
  \[
  \begin{tikzpicture}[xscale=2.3,yscale=-1.2]
    \node (A0_0) at (0, 0) {$\overline{M}_{1,3}$};
    \node (A0_1) at (1, 0) {$\overline{M}_{1,3}$};
    \node (A1_0) at (0, 1) {$\overline{M}_{1,2}$};
    \node (A1_1) at (1, 1) {$\overline{M}_{1,2}$};
    \path (A0_0) edge [->]node [auto] {$\scriptstyle{\phi}$} (A0_1);
    \path (A0_0) edge [->]node [auto,swap] {$\scriptstyle{\pi_{j}}$} (A1_0);
    \path (A0_1) edge [->]node [auto] {$\scriptstyle{\pi_{i}}$} (A1_1);
    \path (A1_0) edge [->]node [auto] {$\scriptstyle{\overline{\phi}}$} (A1_1);
  \end{tikzpicture}
  \]
we get an automorphism $\overline{\phi}$ of $\overline{M}_{1,2}$ which does not preserve $\Delta_{0,2}$.
\end{proof}

\begin{Lemma}\cite[Corollary 0.12]{GKM}\label{Mg}
Any automorphism of $\overline{M}_{g}$ preserves the boundary.
\end{Lemma}
\begin{proof}
Let $\lambda$ be the Hodge class on $\overline{M}_{g}$. It is known that $\lambda$ induces a birational morphism $f:\overline{M}_{g}\rightarrow X$ on a projective variety whose exceptional locus is the boundary $\partial\overline{M}_{g}$, see \cite{Ru}.\\
Assume that there exists an automorphism $\phi:\overline{M}_{g}\rightarrow\overline{M}_{g}$ which does not preserve the boundary. Then there is a point $[C]\in\partial \overline{M}_{g}$ such that $\phi([C]) = [C^{'}]\in M_{g}$.\\
Now $f\circ\phi$ is a birational morphism whose exceptional locus is $\phi^{-1}(\partial\overline{M}_{g})$, and by the assumption on $\phi$ we have $\phi^{-1}(\partial\overline{M}_{g})\cap M_{g}\neq \emptyset$. So we construct a big line bundle on $\overline{M}_{g}$ whose exceptional locus is not contained in the boundary and this contradicts Theorem \ref{GKM}.
\end{proof}

\begin{Proposition}\label{autmg}
For any $g\geq 2$ the only automorphism of $\overline{M}_{g}$ is the identity.
\end{Proposition}
\begin{proof}
Let $\phi$ be an automorphism of $\overline{M}_{g}$. By Lemma \ref{Mg} $\phi$ restricts to an automorphisms $\phi_{|M_{g}}$ of $M_{g}$. If $g\geq 3$ by Royden's theorem \cite[Theorem 6.1]{Mok} $\phi_{|M_{g}}$ is the identity, then $\phi = Id_{\overline{M}_{g}}$.\\
If $g = 2$ the canonical divisor $K_{C}$ of a smooth genus $2$ curve induces a degree $2$ morphism on $\mathbb{P}^{1}$ branched in $6$ points. So we have a morphism 
$$f:M_{2}\rightarrow M_{0,6}/S_{6} \cong M_{0,6}^{un},\: \phi\mapsto \tilde{\phi},$$ 
and since from a $6$-pointed smooth rational curve we can reconstruct the corresponding genus $2$ curve $f$ is indeed an isomorphism. Then $\phi$ induces an automorphism $\tilde{\phi}$ of $M_{0,6}^{un}$, again by \cite[Theorem 6.1]{Mok} we have $\tilde{\phi} = Id_{M_{0,6}^{un}}$ and therefore $\phi = Id_{\overline{M}_{2}}$.
\end{proof}

\begin{Proposition}\label{Mg1}
For any $g\geq 2$ the only automorphism of $\overline{M}_{g,1}$ is the identity. Furthermore $\Aut(\overline{M}_{1,3})\cong S_{3}$.
\end{Proposition}
\begin{proof}
Let $\phi:\overline{M}_{g,1}\rightarrow\overline{M}_{g,1}$ be an automorphism. By Theorem \ref{GKM} the fibration $$\pi_{1}\circ\phi:\overline{M}_{g,1}\rightarrow\overline{M}_{g}$$ 
factors through a forgetful morphism which is necessarily $\pi_{1}$. We have a commutative diagram
  \[
  \begin{tikzpicture}[xscale=2.3,yscale=-1.2]
    \node (A0_0) at (0, 0) {$\overline{M}_{g,1}$};
    \node (A0_1) at (1, 0) {$\overline{M}_{g,1}$};
    \node (A1_0) at (0, 1) {$\overline{M}_{g}$};
    \node (A1_1) at (1, 1) {$\overline{M}_{g}$};
    \path (A0_0) edge [->]node [auto] {$\scriptstyle{\phi}$} (A0_1);
    \path (A0_0) edge [->]node [auto,swap] {$\scriptstyle{\pi_{1}}$} (A1_0);
    \path (A0_1) edge [->]node [auto] {$\scriptstyle{\pi_{1}}$} (A1_1);
    \path (A1_0) edge [->]node [auto] {$\scriptstyle{\overline{\phi}}$} (A1_1);
  \end{tikzpicture}
  \]
so the morphism $\phi$ maps the fiber of $\pi_{1}$ over $[C]$ to the fiber of $\pi_{1}$ over $[C^{'}] := \overline{\phi}([C])$. Now we distinguish two cases.
\begin{itemize}
\item[-] If $g > 2$ then $\pi_{1}^{-1}([C])$ is a smooth genus $g$ curve, so it is automorphisms-free. Let $[C],[C^{'}]\in \overline{M}_{g}$ be two general points, then $\pi_{1}^{-1}([C])\cong C$, $\pi_{1}^{-1}([C^{'}])\cong C^{'}$ and
$$\phi_{|\pi_{1}^{-1}([C])}:C\rightarrow C^{'}$$
is an isomorphism. So $C^{'}\cong C$, $[C^{'}] := \overline{\phi}([C]) = [C]$ and $\overline{\phi} = Id_{\overline{M}_{g}}$. We are thus reduced to a commutative triangle  
  \[
  \begin{tikzpicture}[xscale=2.3,yscale=-1.2]
    \node (A0_0) at (0, 0) {$\overline{M}_{g,1}$};
    \node (A0_2) at (2, 0) {$\overline{M}_{g,1}$};
    \node (A1_1) at (1, 1) {$\overline{M}_{g}$};
    \path (A0_0) edge [->]node [auto,swap] {$\scriptstyle{\pi_{1}}$} (A1_1);
    \path (A0_2) edge [->]node [auto] {$\scriptstyle{\pi_{1}}$} (A1_1);
    \path (A0_0) edge [->]node [auto] {$\scriptstyle{\phi}$} (A0_2);
  \end{tikzpicture}
  \]
and for any $[C]\in\overline{M}_{g}$ the restriction of $\phi$ to the fiber of $\pi_{1}$ defines an automorphism of the fiber. Since $g > 2$ we conclude that $\phi$ is the identity on the general fiber of $\pi_{1}$ so it has to be the identity on $\overline{M}_{g,1}$.
\item[-] Consider now the case $g = 2$. Let $\phi:\overline{M}_{2,1}\rightarrow\overline{M}_{2,1}$ be an automorphism. As usual we have a commutative diagram 
\[
  \begin{tikzpicture}[xscale=2.3,yscale=-1.2]
    \node (A0_0) at (0, 0) {$\overline{M}_{2,1}$};
    \node (A0_1) at (1, 0) {$\overline{M}_{2,1}$};
    \node (A1_0) at (0, 1) {$\overline{M}_{2}$};
    \node (A1_1) at (1, 1) {$\overline{M}_{2}$};
    \path (A0_0) edge [->]node [auto] {$\scriptstyle{\phi}$} (A0_1);
    \path (A0_0) edge [->]node [auto,swap] {$\scriptstyle{\pi_{1}}$} (A1_0);
    \path (A0_1) edge [->]node [auto] {$\scriptstyle{\pi_{1}}$} (A1_1);
    \path (A1_0) edge [->]node [auto] {$\scriptstyle{\overline{\phi}}$} (A1_1);
  \end{tikzpicture}
  \]
The boundary of $\overline{M}_{2,1}$ has two codimension one components parametrizing curves whose dual graphs are 
$$
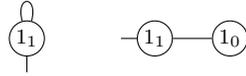
\begin{tabular}{ccc}
\begin{tikzpicture}[scale=.5,inner sep=10][baseline]
      \path(0,0) ellipse (2 and 1);
      \tikzstyle{level 1}=[counterclockwise from=-90,level distance=9mm,sibling angle=0]
      \ndo (A0) at (0:1) {$\scriptstyle{1_{1}}$} child;

      \draw (A0) .. controls +(75.000000:1.2) and +(105.000000:1.2) .. (A0);
    \end{tikzpicture}
&
\qquad
&
\begin{tikzpicture}[scale=.5,inner sep=10][baseline]
      \path(0,0) ellipse (2 and 1);
      \ndo (A0) at (0:1) {$\scriptstyle{1_{0}}$};
      \tikzstyle{level 1}=[counterclockwise from=180,level distance=9mm,sibling angle=0]
      \ndo (A1) at (180:1) {$\scriptstyle{1_{1}}$} child;

      \path (A0) edge [bend left=0.000000] (A1);
    \end{tikzpicture}
\end{tabular}
$$
Similarly the boundary of $\overline{M}_{2}$ has two irreducible components parametrizing curves with dual graphs
$$
\begin{tabular}{ccc}
\begin{tikzpicture}[scale=.5,inner sep=10][baseline]
      \path(0,0) ellipse (2 and 1);
      \ndo (A0) at (0:1) {$\scriptstyle{1_{0}}$};

      \draw (A0) .. controls +(-15.000000:1.2) and +(15.000000:1.2) .. (A0);
    \end{tikzpicture}
&
\qquad
&
\begin{tikzpicture}[scale=.5,inner sep=10][baseline]
      \path(0,0) ellipse (2 and 1);
      \ndo (A0) at (0:1) {$\scriptstyle{1_{0}}$};
      \ndo (A1) at (180:1) {$\scriptstyle{1_{0}}$};

      \path (A0) edge [bend left=0.000000] (A1);
    \end{tikzpicture}
\end{tabular}
$$
\end{itemize}
Clearly $\pi_{1}(\Delta_{irr,1}) = \Delta_{irr}$ and $\pi_{1}(\Delta_{1,1}) = \Delta_{1}$. Suppose that $\phi$ maps either the class of a nodal curve or the class of the union of two elliptic curves to the class of smooth genus $2$ curve then $\overline{\phi}$ has to do the same, and this contradicts Lemma \ref{Mg}.\\
Then $\phi$ maps an open subset of $\partial\overline{M}_{1,2}$ to an open subset of $\partial\overline{M}_{1,2}$ and both these open sets has to intersect the irreducible components of $\partial\overline{M}_{1,2}$. Now the continuity of $\phi$ is enough to conclude that $\phi$ preserves the boundary of $\overline{M}_{2,1}$.\\
Then $\phi$ restrict to an automorphism $M_{2,1}\rightarrow M_{2,1}$. By \cite[Theorem 6.1]{Mok} the only automorphism of $M_{2,1}$ is the identity. Finally $\phi_{|M_{2,1}} = Id_{M_{2,1}}$ implies $\phi = Id_{\overline{M}_{2,1}}$.\\
Consider now the case $g = 1, n = 3$. By Lemma \ref{g1} there exists a factorization $\pi_{i}\circ\phi^{-1} = \overline{\phi^{-1}}\circ\pi_{j_{i}}$, furthermore by Lemma \ref{fl} this factorization is unique. So we have a well defined morphism
$$\chi:\Aut(\overline{M}_{1,3})\rightarrow S_{3},\: \phi\mapsto\sigma_{\phi}$$
where 
$$\sigma_{\phi}:\{1,2,3\}\rightarrow\{1,2,3\},\: i\mapsto j_{i}.$$
Let $\phi$ be an automorphism of $\overline{M}_{1,3}$ inducing the trivial permutation. Then we have three commutative diagrams
  \[
  \begin{tikzpicture}[xscale=2.3,yscale=-1.2]
    \node (A0_0) at (0, 0) {$\overline{M}_{1,3}$};
    \node (A0_1) at (1, 0) {$\overline{M}_{1,3}$};
    \node (A1_0) at (0, 1) {$\overline{M}_{1,2}$};
    \node (A1_1) at (1, 1) {$\overline{M}_{1,2}$};
    \path (A0_0) edge [->]node [auto] {$\scriptstyle{\phi}$} (A0_1);
    \path (A1_0) edge [->]node [auto] {$\scriptstyle{\overline{\phi}}$} (A1_1);
    \path (A0_1) edge [->]node [auto] {$\scriptstyle{\pi_{i}}$} (A1_1);
    \path (A0_0) edge [->]node [auto,swap] {$\scriptstyle{\pi_{i}}$} (A1_0);
  \end{tikzpicture}
  \]
Let $[C,x_{1},x_{2}]\in \overline{M}_{1,2}$ be a general point. The fiber $\pi_{i}^{-1}([C,x_{1},x_{2}])$ intersects the boundary divisors $\Delta_{0,2}\subset \overline{M}_{1,3}$ in two points corresponding to curves with the following dual graph
$$
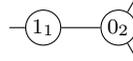
\begin{tikzpicture}[scale=.5,inner sep=10][baseline]
      \path(0,0) ellipse (2 and 1);
      \tikzstyle{level 1}=[counterclockwise from=-60,level distance=9mm,sibling angle=120]
      \ndo (A0) at (0:1) {$\scriptstyle{0_{2}}$} child child;
      \tikzstyle{level 1}=[counterclockwise from=180,level distance=9mm,sibling angle=0]
      \ndo (A1) at (180:1) {$\scriptstyle{1_{1}}$} child;

      \path (A0) edge [bend left=0.000000] (A1);
    \end{tikzpicture}
$$ 
The two points in $\pi_{i}^{-1}([C,x_{1},x_{2}])\cap \Delta_{0,2}$ can be identified with $x_{1},x_{2}$. Now let $[C^{'},x_{1}^{'},x_{2}^{'}]$ be the image of $[C,x_{1},x_{2}]$ via $\overline{\phi}$. Similarly $\pi_{i}^{-1}([C^{'},x_{1}^{'},x_{2}^{'}])\cap \Delta_{0,2} = \{x_{1}^{'},x_{2}^{'}\}$. By Lemma \ref{delta02} we have $\phi(\pi_{i}^{-1}([C,x_{1},x_{2}])\cap \Delta_{0,2}) = \pi_{i}^{-1}([C^{'},x_{1}^{'},x_{2}^{'}])\cap \Delta_{0,2}$ and by Remark \ref{un} $[C^{'},x_{1}^{'},x_{2}^{'}] = [C,x_{1},x_{2}]$ and $\overline{\phi}$ has to be identity.\\
So $\phi$ restrict to an automorphism of the elliptic curve $\pi_{1}^{-1}([C,x_{1},x_{2}])\cong C$ mapping the set $\{x_{1},x_{2}\}$ into itself. On the other hand $\phi$ restricts to an automorphism of the elliptic curve $\pi_{2}^{-1}([C,x_{1},x_{2}])\cong C$ with the same property. Note that $\pi_{2}^{-1}([C,x_{1},x_{2}])\cap\pi_{1}^{-1}([C,x_{1},x_{2}]) = \{x_{1}\}$. The situation is resumed in the following picture:
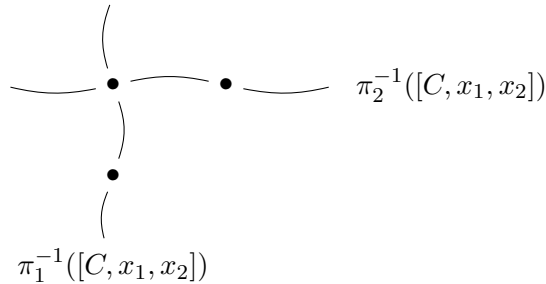
\[
  \begin{tikzpicture}[xscale=2.3,yscale=-1.2]
    \node (A0_1) at (1, 0) {};
    \node (A1_0) at (0, 1) {};
    \node (A1_1) at (1, 1) {$\bullet$};
    \node (A1_2) at (2, 1) {$\bullet$};
    \node (A1_3) at (3, 1) {};
    \node (A1_4) at (4, 1) {$\pi_{2}^{-1}([C,x_{1},x_{2}])$};
    \node (A2_1) at (1, 2) {$\bullet$};
    \node (A3_1) at (1, 3) {$\pi_{1}^{-1}([C,x_{1},x_{2}])$};
    \path (A2_1) edge [-,bend left=15]node [auto] {$\scriptstyle{}$} (A3_1);
    \path (A1_0) edge [-]node [auto] {$\scriptstyle{}$} (A1_0);
    \path (A1_0) edge [-,bend left=15]node [auto] {$\scriptstyle{}$} (A1_1);
    \path (A1_1) edge [-,bend right=15]node [auto] {$\scriptstyle{}$} (A1_2);
    \path (A1_1) edge [-,bend right=15]node [auto] {$\scriptstyle{}$} (A2_1);
    \path (A1_2) edge [-,bend left=15]node [auto] {$\scriptstyle{}$} (A1_3);
    \path (A0_1) edge [-,bend left=15]node [auto] {$\scriptstyle{}$} (A1_1);
  \end{tikzpicture}
\]
Combining these two facts we have that $\phi$ restricts to an automorphism of $\pi_{1}^{-1}([C,x_{1},x_{2}])\cong C$ fixing $x_{1}$ and $x_{2}$. Since $C$ is a general elliptic curve we have that $\phi_{|\pi_{1}^{-1}([C,x_{1},x_{2}])}$ is the identity, and since $[C,x_{1},x_{2}]\in\overline{M}_{1,2}$ is general we conclude that $\phi = Id_{\overline{M}_{1,3}}$. 
\end{proof}

The arguments used in the cases $g\geq 2$ and $g = 1, n\geq 3$ completely fail in the case $g = 1, n = 2$. However, Theorem \ref{wbu} provides a very explicit description of $\overline{M}_{1,2}$ which allows us to describe its automorphism group. Since $\overline{M}_{1,2}$ is a toric surface we know that $(\mathbb{C}^{*})^{2}\subseteq\Aut(\overline{M}_{1,2})$.

\begin{Remark}\label{imp}
The automorphisms of $\mathbb{P}(a_{0},\dots,a_{n})$ are the automorphisms of the graded $k$-algebra $S = k[x_{0},\dots,x_{n}]$.  In particular the automorphisms of $\mathbb{P}(1,2,3)$ are of the form 
$$
\begin{array}{l}
x_{0}\mapsto \alpha_{0}x_{0},\\
x_{1}\mapsto \alpha_{1}x_{0}^{2}+\beta_{1}x_{1},\\
x_{2}\mapsto \alpha_{2}x_{0}^{3}+\beta_{2}x_{0}x_{1}+\gamma_{2}x_{2},
\end{array}
$$
and the the automorphisms of $\mathbb{P}(1,2,3)$ fixing $[1:0:0]$ are of the form 
$$
\begin{array}{l}
x_{0}\mapsto \alpha_{0}x_{0},\\
x_{1}\mapsto \beta_{1}x_{1},\\
x_{2}\mapsto \beta_{2}x_{0}x_{1}+\gamma_{2}x_{2},
\end{array}
$$
with $\alpha_{0},\beta_{1},\gamma_{2}\in k^{*}$ and $\beta_{2}\in k$. The composition law in this group is given by
$$(\alpha_{0},\beta_{1},\beta_{2},\gamma_{2})\ast (\alpha_{0}^{'},\beta_{1}^{'},\beta_{2}^{'},\gamma_{2}^{'}) = (\alpha_{0}\alpha_{0}^{'},\beta_{1}\beta_{1}^{'},\alpha_{0}\beta_{1}\beta_{2}^{'}+\beta_{2}\gamma_{2}^{'},\gamma_{2}\gamma_{2}^{'}).$$
This remark highlights why the automorphisms of the coarse moduli space $\overline{M}_{g,n}$ in general should be different from the automorphisms of the stack $\overline{\mathcal{M}}_{g,n}$. It is well known that $\overline{M}_{1,1}\cong \mathbb{P}^{1}$ and $\overline{\mathcal{M}}_{1,1}\cong \mathbb{P}(4,6)$. Clearly $\mathbb{P}^{1}\cong \mathbb{P}(4,6)$ as varieties, however they are not isomorphic as stacks, indeed $\mathbb{P}(4,6)$ has two stacky points with stabilizers $\mathbb{Z}_{4}$ and $\mathbb{Z}_{6}$. These two points are fixed by any automorphism of $\mathbb{P}(4,6)$ while they are indistinguishable from any other point on the coarse moduli space $\overline{M}_{1,1}$. By the previous description the automorphisms of $\overline{\mathcal{M}}_{1,1}\cong \mathbb{P}(4,6)$ are of the form
$$
\begin{array}{l}
x_{0}\mapsto \alpha_{0}x_{0},\\
x_{1}\mapsto \beta_{1}x_{1},
\end{array}
$$
with $\alpha_{0},\alpha_{1}\in k^{*}$.
\end{Remark}

\begin{Proposition}\label{autm12}
The automorphism group of $\overline{M}_{1,2}$ is isomorphic to $(\mathbb{C}^{*})^{2}$.
\end{Proposition} 
\begin{proof}
By Theorem \ref{wbu} $\overline{M}_{1,2}$ is a weighted blow up of $\mathbb{P}(1,2,3)$ in $[1:0:0]$. Let $\phi$ be an automorphism of $\overline{M}_{1,2}$. Then we have a commutative diagram 
  \[
  \begin{tikzpicture}[xscale=2.3,yscale=-1.2]
    \node (A0_0) at (0, 0) {$\overline{M}_{1,2}$};
    \node (A0_1) at (1, 0) {$\overline{M}_{1,2}$};
    \node (A1_0) at (0, 1) {$\overline{M}_{1,1}$};
    \node (A1_1) at (1, 1) {$\overline{M}_{1,1}$};
    \path (A0_0) edge [->]node [auto] {$\scriptstyle{\phi}$} (A0_1);
    \path (A0_0) edge [->]node [auto,swap] {$\scriptstyle{\pi_{1}}$} (A1_0);
    \path (A0_1) edge [->]node [auto] {$\scriptstyle{\pi_{1}}$} (A1_1);
    \path (A1_0) edge [->]node [auto] {$\scriptstyle{\overline{\phi}}$} (A1_1);
  \end{tikzpicture}
  \]
and $\phi$ has to map fibers of $\pi_{1}$ on fibers of $\pi_{1}$. Let $f:\overline{M}_{1,2}\rightarrow\mathbb{P}(1,2,3)$ be the contraction described in Theorem \ref{wbu}. 
Let $p_{4},p_{6}\in\Delta_{0,2}$ be the two singular points on the exceptional divisor, and let $q_{4},q_{6}\in M_{1,2}$ be the other two singular points. Since $\Delta_{0,2}$ is the only rational contractible curve in $\overline{M}_{1,2}$ it has to be stabilized by $\phi$, furthermore $\phi(p_{4}) = p_{4}$ and $\phi(p_{6}) = p_{6}$. Let $F_{6}$ be the fiber of $\pi_{1}$ trough $p_{6},q_{6}$ and let $F_{4}$ be the fiber of $\pi_{1}$ trough $p_{4},q_{4}$. Since $\phi(q_{4}) = q_{4}$ and $\phi(q_{6}) = q_{6}$ we get $\phi(F_{4}) = F_{4}$ and $\phi(F_{6}) = F_{6}$.\\
We denote by $L_{6}:=f(F_{6}), L_{4}:=f(F_{4})$ the images via $f$ of $F_{6}$ and $F_{4}$ respectively. The automorphism $\phi$ induces via $f$ an automorphism $\tilde{\phi}$ of $\mathbb{P}(1,2,3)$ fixing $[1:0:0]$ and stabilizing $L_{6}, L_{4}$. Let $G$ be the group
$$G:=\{g\in\Aut(\mathbb{P}(1,2,3))\: |\: g([1:0:0]) = [1:0:0],\; g(L_{4}) = L_{4},\; g(L_{6}) = L_{6}\},$$
and consider the morphism of groups
$$\chi:\Aut(\overline{M}_{1,2})\rightarrow G, \: \phi\mapsto\tilde{\phi}.$$
Clearly $\chi$ is injective.\\
Let $x_{0},x_{1},x_{2}$ be the coordinates on $\mathbb{P}(1,2,3)$. Note that the fiber $F_{6}$ corresponding to the Weierstrass curve $C_{6}$ and the fiber $F_{4}$ corresponding to the Weierstrass curve $C_{4}$ are mapped by $f$ in the curves $L_{6} = \{x_{1}=0\}$ and $L_{4} = \{x_{2}=0\}$. By Remark \ref{imp} the automorphisms of $\mathbb{P}(1,2,3)$ fixing $[1:0:0]$ are of the form 
$$
\begin{array}{l}
x_{0}\mapsto \alpha_{0}x_{0},\\
x_{1}\mapsto \beta_{1}x_{1},\\
x_{2}\mapsto \beta_{2}x_{0}x_{1}+\gamma_{2}x_{2},
\end{array}
$$
and forcing an automorphism to stabilize $L_{4}$ and $L_{6}$ gives $\beta_{2} = 0$. Then the automorphisms in $G$ are of the form 
$$
\begin{array}{l}
x_{0}\mapsto \alpha_{0}x_{0},\\
x_{1}\mapsto \beta_{1}x_{1},\\
x_{2}\mapsto \gamma_{2}x_{2},
\end{array}
$$
where $\alpha_{0},\beta_{1},\gamma_{2}\in\mathbb{C}^{*}$, so $G\cong(\mathbb{C}^{*})^{2}$. The automorphism $\tilde{\phi}(x_{0},x_{1},x_{2}) = (\alpha_{0}x_{0},\beta_{1}x_{1},\gamma_{2}x_{2})$ is $\chi(\phi)$ where $\phi$ is the automorphism of $\overline{M}_{1,2}$ acting as $\phi(x,y,a,b) = (\alpha_{0}x,\beta_{1}a,\gamma_{2}b)$. Consider the fibration $\overline{M}_{1,2}\rightarrow\overline{M}_{1,1}$. The automorphism $\phi$ acts on the couple $(a,b)$ as an automorphism of $\overline{M}_{1,1}\cong\mathbb{P}^{1}$ and multiplying by $\alpha_{0}$ on the fibers. So $\chi$ is surjective.
\end{proof}

In order to proceed by induction on $n$ we need the following lemma.

\begin{Lemma}\label{fl}
Let $\phi:\overline{M}_{g,n}\rightarrow\overline{M}_{g,n}$ be an automorphism. For any $j=1,\dots,n$ there exists a commutative diagram
\[
  \begin{tikzpicture}[xscale=2.3,yscale=-1.2] 
    \node (A0_0) at (0, 0) {$\overline{M}_{g,n}$};
    \node (A0_1) at (1, 0) {$\overline{M}_{g,n}$};
    \node (A1_0) at (0, 1) {$\overline{M}_{g,n-1}$};
    \node (A1_1) at (1, 1) {$\overline{M}_{g,n-1}$};
    \path (A0_0) edge [->]node [auto] {$\scriptstyle{\phi}$} (A0_1);
    \path (A0_0) edge [->]node [auto,swap] {$\scriptstyle{\pi_{i}}$} (A1_0);
    \path (A0_1) edge [->]node [auto] {$\scriptstyle{\pi_{j}}$} (A1_1);
    \path (A1_0) edge [->]node [auto] {$\scriptstyle{\overline{\phi}}$} (A1_1);
  \end{tikzpicture}
  \]
\begin{itemize}
\item[-] The morphism $\overline{\phi}$ is an automorphism of $\overline{M}_{g,n-1}$;
\item[-] the factorization of $\pi_{j}\circ\phi$ is unique for any $j = 1,\dots,n$.
\end{itemize}  
\end{Lemma}
\begin{proof}
The existence of such a diagram is ensured by Theorem \ref{GKM} and Lemma \ref{g1}. Let $[C,x_{1},\dots,x_{n-1}]\in \overline{M}_{g,n-1}$ be a point, the automorphism $\phi^{-1}$ maps isomorphically the fiber of $\pi_{j}$ over $[C,x_{1},\dots,x_{n-1}]$ to a fiber $F$ of $\pi_{i}$, so $\pi_{i}(F)=[C^{'},x_{1}^{'},\dots ,x_{n-1}^{'}]$ is a point. Define $\overline{\psi}:\overline{M}_{g,n-1}\rightarrow\overline{M}_{g,n-1}$ as $\overline{\psi}([C,x_{1},\dots ,x_{n-1}])= [C^{'},x_{1}^{'},\dots,x_{n-1}^{'}]$. Clearly $\overline{\psi}$ is the inverse of $\overline{\phi}$.\\
Suppose that $\pi_{j}\circ\phi$ admits two factorizations $\overline{\phi}_{1}\circ\pi_{i}$ and $\overline{\phi}_{2}\circ\pi_{h}$. Then the equality $\overline{\phi}_{1}\circ\pi_{i}([C,x_{1},\dots,x_{n}]) = \overline{\phi}_{2}\circ\pi_{h}([C,x_{1},\dots,x_{n}])$ for any $[C,x_{1},\dots,x_{n}]\in\overline{M}_{g,n}$ implies $$\overline{\phi}_{1}([C,y_{1},\dots,y_{n-1}]) = \overline{\phi}_{2}([C,y_{1},\dots,y_{n-1}])$$ 
for any $[C,y_{1},\dots,y_{n-1}]\in\overline{M}_{g,n-1}$.
Now $\overline{\phi}_{1} = \overline{\phi}_{2}$ implies $\overline{\phi}_{1}\circ\pi_{i} = \overline{\phi}_{1}\circ\pi_{h}$ and since $\overline{\phi}_{1}$ is an isomorphism we have $\pi_{i} = \pi_{h}$. 
\end{proof}

At this point we can prove the general theorem by induction on $n$.

\begin{Theorem}\label{aut}
The automorphism group of $\overline{M}_{g,n}$ is isomorphic to the symmetric group on $n$ elements $S_{n}$
$$\Aut(\overline{M}_{g,n})\cong S_{n}$$
for any $g,n$ such that $2g-2+n\geq 3$.
\end{Theorem}
\begin{proof}
Proposition \ref{Mg1} gives the cases $g\geq 2, n = 1$ and $g = 1, n = 3$. We proceed by induction on $n$. Let $\phi$ be an automorphism of $\overline{M}_{g,n}$, consider the composition $\pi_{i}\circ\phi^{-1}$. By Theorem \ref{GKM} there exists a factorization $\pi_{i}\circ\phi^{-1} = \overline{\phi^{-1}}\circ\pi_{j_{i}}$, furthermore by Lemma \ref{fl} this factorization is unique. So we have a well defined map
$$\chi:\Aut(\overline{M}_{g,n})\rightarrow S_{n},\: \phi\mapsto\sigma_{\phi}$$
where 
$$\sigma_{\phi}:\{1,\dots,n\}\rightarrow\{1,\dots,n\},\: i\mapsto j_{i}.$$
In order to prove that $\sigma_{\phi}$ is actually a permutation we prove that it is injective. Suppose to have $\sigma_{\phi}(i) = j_{i} = \sigma_{\phi}(h)$. This means that $\phi^{-1}$ defines an isomorphism between the fibers of $\pi_{j_{i}}$ and $\pi_{i}$, but also between the fibers of $\pi_{j_{i}}$ and $\pi_{h}$. This forces $\pi_{i} = \pi_{h}$.\\ 
We now prove that the map $\chi$ is a morphism of groups. Let $\phi,\psi\in\overline{M}_{g,n}$ be two automorphisms. The fibration $\pi_{i}\circ\psi^{-1}$ factorizes through $\pi_{j_{i}}$ and similarly $\pi_{j_{i}}\circ\phi^{-1}$ factorizes though $\pi_{h_{i}}$. By uniqueness of the factorization $\pi_{i}\circ (\psi^{-1}\circ\phi^{-1})$ factorizes through $\pi_{h_{i}}$ also. The situation is resumed in the following commutative diagram
  \[
  \begin{tikzpicture}[xscale=2.3,yscale=-1.2]
    \node (A0_0) at (0, 0) {$\overline{M}_{g,n}$};
    \node (A0_1) at (1, 0) {$\overline{M}_{g,n}$};
    \node (A0_2) at (2, 0) {$\overline{M}_{g,n}$};
    \node (A1_0) at (0, 1) {$\overline{M}_{g,n-1}$};
    \node (A1_1) at (1, 1) {$\overline{M}_{g,n-1}$};
    \node (A1_2) at (2, 1) {$\overline{M}_{g,n-1}$};
    \path (A0_0) edge [->]node [auto] {$\scriptstyle{\phi^{-1}}$} (A0_1);
    \path (A0_1) edge [->]node [auto] {$\scriptstyle{\psi^{-1}}$} (A0_2);
    \path (A1_0) edge [->]node [auto] {$\scriptstyle{\overline{\phi^{-1}}}$} (A1_1);
    \path (A0_2) edge [->]node [auto] {$\scriptstyle{\pi_{i}}$} (A1_2);
    \path (A1_1) edge [->]node [auto] {$\scriptstyle{\overline{\psi^{-1}}}$} (A1_2);
    \path (A1_0) edge [->,bend left=25]node [auto,swap] {$\scriptstyle{\overline{(\phi\circ\psi)^{-1}}}$} (A1_2);
    \path (A0_0) edge [->]node [auto,swap] {$\scriptstyle{\pi_{h_{i}}}$} (A1_0);
    \path (A0_1) edge [->]node [auto] {$\scriptstyle{\pi_{j_{i}}}$} (A1_1);
  \end{tikzpicture}
  \]
This means that $\sigma_{\psi}(i) = j_{i}$, $\sigma_{\phi}(j_{i}) = h_{i}$ and $\sigma_{\phi\circ\psi}(i) = h_{i}$. Then $\sigma_{\phi\circ\psi}(i) = \sigma_{\phi}(j_{i}) = \sigma_{\phi}(\sigma_{\psi}(i))$, that is $\chi(\phi\circ\psi) = \chi(\phi)\circ\chi(\psi)$.\\
Since any permutation of the marked points induces an automorphism of $\overline{M}_{g,n}$ the morphism $\chi$ is surjective. Now we compute its kernel.\\
Let $\phi\in\Aut(\overline{M}_{g,n})$ be an automorphism such that $\chi(\phi)$ is the identity, that is for any $i = 1,\dots,n$ the fibration $\pi_{i}\circ\phi$ factors through $\pi_{i}$ and we have $n$ commutative diagrams
  \[
  \begin{tikzpicture}[xscale=2.3,yscale=-1.2]
    \node (A0_0) at (0, 0) {$\overline{M}_{g,n}$};
    \node (A0_1) at (1, 0) {$\overline{M}_{g,n}$};
    \node (A1_0) at (0, 1) {$\overline{M}_{g,n-1}$};
    \node (A1_1) at (1, 1) {$\overline{M}_{g,n-1}$};
    \path (A0_0) edge [->]node [auto] {$\scriptstyle{\phi}$} (A0_1);
    \path (A1_0) edge [->]node [auto,swap] {$\scriptstyle{\overline{\phi}_{1}}$} (A1_1);
    \path (A0_1) edge [->]node [auto] {$\scriptstyle{\pi_{1}}$} (A1_1);
    \path (A0_0) edge [->]node [auto,swap] {$\scriptstyle{\pi_{1}}$} (A1_0);
  \end{tikzpicture}
  \quad
  \cdots
  \quad
  \begin{tikzpicture}[xscale=2.3,yscale=-1.2]
    \node (A0_0) at (0, 0) {$\overline{M}_{g,n}$};
    \node (A0_1) at (1, 0) {$\overline{M}_{g,n}$};
    \node (A1_0) at (0, 1) {$\overline{M}_{g,n-1}$};
    \node (A1_1) at (1, 1) {$\overline{M}_{g,n-1}$};
    \path (A0_0) edge [->]node [auto] {$\scriptstyle{\phi}$} (A0_1);
    \path (A1_0) edge [->]node [auto,swap] {$\scriptstyle{\overline{\phi}_{n}}$} (A1_1);
    \path (A0_1) edge [->]node [auto] {$\scriptstyle{\pi_{n}}$} (A1_1);
    \path (A0_0) edge [->]node [auto,swap] {$\scriptstyle{\pi_{n}}$} (A1_0);
  \end{tikzpicture}
  \]
By Lemma \ref{fl} the morphisms $\overline{\phi}_{i}$ are automorphisms of $\overline{M}_{g,n-1}$ and by induction hypothesis $\overline{\phi}_{1},\dots,\overline{\phi}_{n}$ act on $\overline{M}_{g,n-1}$ as permutations.\\ 
The action of $\overline{\phi}_{i}$ on the marked points $x_{1},\dots,x_{i-1},x_{i+1},\dots,x_{n}$ has to lift to the same automorphism $\phi$ for any $i=1,\dots,n$. So the actions of $\overline{\phi}_{1},\dots,\overline{\phi}_{n}$ have to be compatible and this implies $\overline{\phi}_{i} = Id_{\overline{M}_{g,n-1}}$ for any $i=1,\dots,n$. We distinguish two cases.
\begin{itemize}
\item[-] Assume $g\geq 3$. It is enough to observe that $\phi$ restricts to an automorphism of the fibers of $\pi_{1}$. Then $\phi$ restricts to the identity on the general fiber of $\pi_{1}$, so $\phi = Id_{\overline{M}_{g,n}}$.
\item[-] Assume $g = 1,2$. Note that $\phi$ restricts to an automorphism of the fibers of $\pi_{1}$ and $\pi_{2}$. So $\phi$ defines an automorphism of the fiber of $\pi_{1}$ with at least two fixed points in the case $g = 1, n\geq 3$ and one fixed point in the case $g = 2, n\geq 2$. Since the general $2$-pointed genus $1$ curve and the general $1$-pointed genus $2$ curves have no non trivial automorphisms we conclude as before that $\phi$ restricts to the identity on the general fiber of $\pi_{1}$, so $\phi = Id_{\overline{M}_{g,n}}$.  
\end{itemize}
This proves that $\chi$ is injective and defines an isomorphism between $\Aut(\overline{M}_{g,n})$ and $S_{n}$.
\end{proof}

We want to use the techniques developed in this section to recover \cite[Theorem 4.3]{BM2}. The moduli spaces $\overline{M}_{0,4}$ is isomorphic to the projective line $\mathbb{P}^{1}$ while $\overline{M}_{0,5}$ is the blow-up of $\mathbb{P}^{2}$ in four points in general position. The following is well known but we want to give a proof following the argument used in Proposition \ref{Mg1}.

\begin{Proposition}\label{M05}
The automorphism group of $\overline{M}_{0,5}$ is isomorphic to $S_{5}$.
\end{Proposition}
\begin{proof}
It is well known that any fibration $\overline{M}_{0,5}\rightarrow\overline{M}_{0,4}$ factorizes through a forgetful morphism, see for instance \cite{BM2}. This yields a surjective morphism of groups
$$\chi:\Aut(\overline{M}_{0,5})\rightarrow S_{5}$$
exactly as in Theorem \ref{aut}. Let $\phi$ be an automorphism of $\overline{M}_{0,5}$ inducing the trivial permutation. Then we get five commutative diagrams 
  \[
  \begin{tikzpicture}[xscale=2.3,yscale=-1.2]
    \node (A0_0) at (0, 0) {$\overline{M}_{0,5}$};
    \node (A0_1) at (1, 0) {$\overline{M}_{0,5}$};
    \node (A1_0) at (0, 1) {$\overline{M}_{0,4}$};
    \node (A1_1) at (1, 1) {$\overline{M}_{0,4}$};
    \path (A0_0) edge [->]node [auto] {$\scriptstyle{\phi}$} (A0_1);
    \path (A1_0) edge [->]node [auto] {$\scriptstyle{\overline{\phi}_{i}}$} (A1_1);
    \path (A0_1) edge [->]node [auto] {$\scriptstyle{\pi_{i}}$} (A1_1);
    \path (A0_0) edge [->]node [auto,swap] {$\scriptstyle{\pi_{i}}$} (A1_0);
  \end{tikzpicture}
  \]
for $i = 1,\dots,5$. The fiber of $\pi_{i}$ on $[C,x_{1},\dots,x_{4}]\in \overline{M}_{0,4}$ intersects the boundary $\partial\overline{M}_{0,4}$ in four points corresponding to $x_{1},\dots,x_{4}$.\\ 
Consider $[C^{'},x_{1}^{'},\dots,x_{4}^{'}]:=\overline{\phi}_{i|[C,x_{1},\dots,x_{4}]}([C,x_{1},\dots,x_{4}])$. The points in $\pi_{i}^{-1}([C,x_{1},\dots,x_{4}])\cap\partial\overline{M}_{0,4}$ and in $\pi_{i}^{-1}([C^{'},x_{1}^{'},\dots,x_{4}^{'}])\cap\partial\overline{M}_{0,4}$  lie on $(-1)$-curves, so the automorphism $\phi$ maps the fiber of $\pi_{i}$ over $[C,x_{1},\dots,x_{4}]$ to the fiber of $\pi_{i}$ over $[C^{'},x_{1}^{'},\dots,x_{4}^{i}]$ sending the set $\{x_{1},\dots,x_{4}\}$ to the set $\{x_{1}^{'},\dots,x_{4}^{'}\}$. Then $\overline{\phi}_{1},\dots,\overline{\phi}_{5}$ act as permutations of the marking and since they come from the same automorphism $\phi$ they have to be compatible. This forces $\overline{\phi}_{1}= \dots =\overline{\phi}_{5} = Id_{\overline{M}_{0,4}}$.\\
Let $[C,x_{1},\dots,x_{4}]\in \overline{M}_{0,4}$ be a general point. The automorphism $\phi$ restricts to an automorphism of the fiber $\pi_{1}^{-1}([C,x_{1},\dots,x_{4}])\cong\mathbb{P}^{1}$ stabilizing the subscheme $\{x_{1},\dots,x_{4}\}\subset\pi_{1}^{-1}([C,x_{1},\dots,x_{4}])$. Since $x_{1},\dots,x_{4}$ are general points of $C$ they have a cross-ratio different from the cross-ratio of each permutation. This means that $\phi_{|C}$ is an automorphism of $\mathbb{P}^{1}$ fixing four points. So $\phi$ restricts to the identity on the general fiber of $\pi_{1}$ and this forces $\phi = Id_{\overline{M}_{0,5}}$.
\end{proof}

\begin{Remark}
The moduli space $\overline{M}_{0,5}$ is isomorphic to a Del Pezzo surface of degree $5$, by Proposition \ref{M05} we recover that the automorphism group of such a surface is $S_{5}$. For a direct proof of this classical fact which does not use the theory of moduli spaces see \cite[Section 3]{DI}. 
\end{Remark}

Now with the same argument of Theorem \ref{aut} we can prove the following:

\begin{Theorem}
The automorphism group of $\overline{M}_{0,n}$ is isomorphic to the symmetric group on $n$ elements $S_{n}$
$$\Aut(\overline{M}_{0,n})\cong S_{n}$$
for any $n\geq 5$.
\end{Theorem}
\begin{proof}
The step zero of the induction is Proposition \ref{M05}. As usual we have a surjective morphism of groups
$$\chi:\overline{M}_{0,n}\rightarrow S_{n}.$$
Proceeding as in the proof of Theorem \ref{aut} we get that an automorphism $\phi$ inducing the trivial permutation has to restrict to an automorphism of the fiber of $\pi_{i}:\overline{M}_{0,n}\rightarrow\overline{M}_{0,n-1}$ fixing $k\geq 4$ points. So it has to be the identity on the general fiber of $\pi_{i}$, and therefore also on $\overline{M}_{0,n}$.
\end{proof}

In \cite[Corollary 0.12]{GKM} \textit{Gibney}, \textit{Keel} and \textit{Morrison} proved that any automorphism of $\overline{M}_{g}$ must preserve the boundary.\\
From Theorem \ref{aut} follows immediately that the boundary of $\overline{M}_{g,n}$ has a good behavior under the action of $\Aut(\overline{M}_{g,n})$. The result is even stronger than the preservation of the boundary.

\begin{Corollary}
If $2g-2+n\geq 3$ any automorphism of $\overline{M}_{g,n}$ must preserve all strata of the boundary.
\end{Corollary}
\begin{proof}
Since any automorphism is a permutation the class of a pointed curve $[C,x_{1},\dots,x_{n}]$ is mapped by an automorphism in a class $[C^{'},x_{1}^{'},\dots,x_{n}^{'}]$ representing a pointed curve of the same topological type of the pointed curve $C$.
\end{proof}

\section{Automorphisms of $\overline{\mathcal{M}}_{g,n}$}\label{stack}

Let $\mathcal{X}$ be an algebraic stack over $\mathbb{C}$. A coarse moduli space for $\mathcal{X}$ over $\mathbb{C}$ is a morphism $\pi:\mathcal{X}\rightarrow X$, where $X$ is an algebraic space over $\mathbb{C}$ such that
\begin{itemize}
\item[-] the morphism $\pi$ is universal for morphisms to algebraic spaces,
\item[-] $\pi$ induces a bijection between $|\mathcal{X}|$ and the closed points of $X$, where $|\mathcal{X}|$ denotes the set of isomorphism classes in $\mathcal{X}$. 
\end{itemize}
\begin{Remark}
If $\mathcal{X}$ admits a coarse moduli space $\pi:\mathcal{X}\rightarrow X$ then this is unique up to unique isomorphism.
\end{Remark}
A separated algebraic stack has a coarse moduli space which is a separated algebraic space \cite[Corollary 1.3]{KM}.\\
Let $\mathcal{X}$ be a separated stack admitting a scheme $X$ as coarse moduli space $\pi:\mathcal{X}\rightarrow X$. The map $\pi$ is universal for morphisms in schemes, that is for any morphism $f:\mathcal{X}\rightarrow Y$, with $Y$ scheme, there exists a unique morphisms of schemes $g:X\rightarrow Y$ such that the diagram
  \[
  \begin{tikzpicture}[xscale=2.3,yscale=-1.2]
    \node (A0_0) at (0, 0) {$\mathcal{X}$};
    \node (A0_2) at (2, 0) {$X$};
    \node (A1_1) at (1, 1) {$Y$};
    \path (A0_0) edge [->]node [auto,swap] {$\scriptstyle{f}$} (A1_1);
    \path (A0_2) edge [->]node [auto] {$\scriptstyle{g}$} (A1_1);
    \path (A0_0) edge [->]node [auto] {$\scriptstyle{\pi}$} (A0_2);
  \end{tikzpicture}
  \]
commutes. Now, let $\phi:\mathcal{X}\rightarrow\mathcal{X}$ be an automorphism of the stack $\mathcal{X}$, and consider $\pi\circ\phi:\mathcal{X}\rightarrow X$. Then these exists a unique $\tilde{\phi}$ such that the diagram
  \[
  \begin{tikzpicture}[xscale=2.3,yscale=-1.2]
    \node (A0_0) at (0, 0) {$\mathcal{X}$};
    \node (A0_1) at (1, 0) {$\mathcal{X}$};
    \node (A1_0) at (0, 1) {$X$};
    \node (A1_1) at (1, 1) {$X$};
    \path (A0_0) edge [->]node [auto] {$\scriptstyle{\phi}$} (A0_1);
    \path (A0_0) edge [->]node [auto,swap] {$\scriptstyle{\pi}$} (A1_0);
    \path (A0_1) edge [->]node [auto] {$\scriptstyle{\pi}$} (A1_1);
    \path (A1_0) edge [->]node [auto] {$\scriptstyle{\tilde{\phi}}$} (A1_1);
  \end{tikzpicture}
  \]
commutes. By uniqueness we have $(\tilde{\phi})^{-1} = \tilde{\phi^{-1}}$. So $\tilde{\phi}$ is an automorphisms of $X$, and we get a morphism of groups
$$\Aut(\mathcal{X})\rightarrow \Aut(X),\: \phi\mapsto\tilde{\phi}.$$

\begin{Remark}
Even if $\mathcal{X}$ is a Deligne-Mumford stack with trivial generic stabilizer the above morphism of groups is not necessarily injective. As instance in \cite[Proposition 7.1.1]{ACV} \textit{D. Abramovich}, \textit{A. Corti} and \textit{A. Vistoli} consider a twisted curve $\mathcal{C}$ over an algebraically closed field and its coarse moduli space $C$. They prove that for any node $x\in C$ the stabilizer of a geometric point of $\mathcal{C}$ over $x$ contributes to the automorphism group of $\mathcal{C}$ over $C$.
\end{Remark}

However, since $\overline{\mathcal{M}}_{g,n}$ is a normal, Deligne-Mumford stack, as soon as its general point has trivial stabilizer, the morphism
$$\Aut(\overline{\mathcal{M}}_{g,n})\rightarrow\Aut(\overline{M}_{g,n})$$
is injective. Our next goal is to prove this last statement.

\begin{Proposition}\label{inj}
The morphism of groups
$$\Aut(\overline{\mathcal{M}}_{g,n})\rightarrow\Aut(\overline{M}_{g,n})$$
is injective as soon as the general $n$-pointed genus $g$ curve has no non trivial automorphisms.
\end{Proposition}
\begin{proof}
In \cite[Proposition A.1]{FMN} take $\mathcal{X} = \mathcal{Y} = \overline{\mathcal{M}}_{g,n}$. Since we consider the case when the general $n$-pointed genus $g$ curve has no non trivial automorphisms there is a dense open subscheme $U\subset\overline{M}_{g,n}$ where the canonical map $\overline{\mathcal{M}}_{g,n}\rightarrow\overline{M}_{g,n}$ is an isomorphism. Note that $\overline{\mathcal{M}}_{g,n}$ is an irreducible normal and separated Deligne-Mumford stack, so the hypothesis of \cite[Proposition A.1]{FMN} are satisfied.\\
Let $f:\overline{\mathcal{M}}_{g,n}\rightarrow\overline{\mathcal{M}}_{g,n}$ be an automorphism inducing the identity on the coarse moduli space $\overline{M}_{g,n}$, then there is a $2$-arrow $\alpha:f_{|U}\Longrightarrow Id_{U}$. By \cite[Proposition A.1]{FMN} there exists a unique $2$-arrow $\overline{\alpha}:f\Longrightarrow Id_{\overline{\mathcal{M}}_{g,n}}$ extending $\alpha$. We conclude that $\overline{\alpha}$ is an isomorphism and $f$ is isomorphic to the identity of $\overline{\mathcal{M}}_{g,n}$.
\end{proof}

\begin{Theorem}\label{autstack}
The automorphism group of the stack $\overline{\mathcal{M}}_{g,n}$ is isomorphic to the symmetric group on $n$ elements $S_{n}$
$$\Aut(\overline{\mathcal{M}}_{g,n})\cong S_{n}$$
for any $g,n$ such that $2g-2+n\geq 3$. Furthermore $\Aut(\overline{\mathcal{M}}_{g})$ is trivial for any $g\geq 2$.
\end{Theorem}
\begin{proof}
For any $g,n$ in our range the general point of $\overline{\mathcal{M}}_{g,n}$ has trivial automorphism group. So by Proposition \ref{inj} the morphism of groups
$$\Aut(\overline{\mathcal{M}}_{g,n})\rightarrow\Aut(\overline{M}_{g,n})$$
is injective. By Theorem \ref{aut} and \cite[Theorem 4.3]{BM2} we know that $\Aut(\overline{M}_{g,n})\cong S_{n}$ for the values of $g$ and $n$ we are considering. Since any permutation of the marked points in an automorphism of $\overline{\mathcal{M}}_{g,n}$ we conclude that 
$$\Aut(\overline{\mathcal{M}}_{g,n})\cong\Aut(\overline{M}_{g,n})\cong S_{n}.$$
Since the general curve of genus $g\geq 3$ is automorphisms free the morphism
$$\Aut(\overline{\mathcal{M}}_{g})\rightarrow\Aut(\overline{M}_{g})$$
is injective. We conclude by Proposition \ref{autmg}. In the case $g = 2$ consider the fiber product
  \[
  \begin{tikzpicture}[xscale=3.5,yscale=-1.2]
    \node (A0_0) at (0, 0) {$\overline{\mathcal{M}}_{2,1}\times_{\overline{\mathcal{M}}_{2}}\overline{\mathcal{M}}_{2}\cong\overline{\mathcal{M}}_{2,1}$};
    \node (A0_1) at (1, 0) {$\overline{\mathcal{M}}_{2,1}$};
    \node (A1_0) at (0, 1) {$\overline{\mathcal{M}}_{2}$};
    \node (A1_1) at (1, 1) {$\overline{\mathcal{M}}_{2}$};
    \path (A0_0) edge [->]node [auto] {$\scriptstyle{\psi}$} (A0_1);
    \path (A1_0) edge [->]node [auto] {$\scriptstyle{\phi}$} (A1_1);
    \path (A0_1) edge [->]node [auto] {$\scriptstyle{\pi_{1}}$} (A1_1);
    \path (A0_0) edge [->]node [auto] {$\scriptstyle{}$} (A1_0);
  \end{tikzpicture}
  \]
where $\phi\in\Aut(\overline{\mathcal{M}}_{2})$. Since $\phi$ is an automorphism $\psi$ also is an automorphism. By the previous part of the proof we know that $\Aut(\overline{\mathcal{M}}_{2,1})\cong\Aut(\overline{M}_{2,1})$ is trivial. So $\psi = Id_{\overline{\mathcal{M}}_{2,1}}$ and therefore $\phi = Id_{\overline{\mathcal{M}}_{2}}$.
\end{proof}

As we saw in Proposition \ref{autm12} the case $g = 1, n = 2$ is pathological from the point of view of the automorphisms. Since $\Aut(\overline{M}_{1,2})\cong (\mathbb{C}^{*})^{2}$ the injectivity of the morphism $\Aut(\overline{\mathcal{M}}_{1,2})\rightarrow\Aut(\overline{M}_{1,2})$
does not say to much on $\Aut(\overline{\mathcal{M}}_{1,2})$. 

Since all the automorphisms of $\overline{M}_{1,2}$ are toric we expect them to disappear on the stack. In the following proposition we prove that $\Aut(\overline{\mathcal{M}}_{1,2})$ is trivial exploiting the particular form of its canonical divisor.

\begin{Proposition}\label{stackm12}
The only automorphism of the moduli stack $\overline{\mathcal{M}}_{1,2}$ is the identity.
\end{Proposition}
\begin{proof}
An application of the Grothendieck-Riemann-Roch theorem \cite[Section 3E]{HM} gives the following formula for the canonical class of $\overline{\mathcal{M}}_{1,2}$
$$K_{\overline{\mathcal{M}}_{1,2}} = 13\lambda -2\delta +\psi \in \Pic_{\mathbb{Q}}(\overline{\mathcal{M}}_{1,2}).$$
The Picard group $\Pic_{\mathbb{Q}}(\overline{\mathcal{M}}_{1,2})$ is freely generated by $\lambda$ and the boundary classes, furthermore the following relations hold \cite[Theorem 2.2]{AC}:
$$\delta_{irr} = 12\lambda,\: \psi = 2\lambda + 2\delta_{0,2}.$$
We can write the canonical class in terms of the boundary divisors as
$$K_{\overline{\mathcal{M}}_{1,2}} = \frac{13}{12}\delta_{irr}-2\delta_{irr}-2\delta_{0,2}+\frac{2}{12}\delta_{irr}+2\delta_{0,2} = -\frac{3}{4}\delta_{irr}.$$
Note that $\delta_{irr}$ is a fiber of the forgetful morphism $\pi_{1}:\overline{\mathcal{M}}_{1,2}\rightarrow\overline{\mathcal{M}}_{1,1}$. Any automorphism $\phi$ of $\overline{\mathcal{M}}_{1,2}$ preserves the canonical bundle, that is $\phi^{*}K_{\overline{\mathcal{M}}_{1,2}} = K_{\overline{\mathcal{M}}_{1,2}}$ in $\Pic_{\mathbb{Q}}(\overline{\mathcal{M}}_{1,2})$.\\ 
Since $K_{\overline{\mathcal{M}}_{1,2}}$ is a multiple of the fiber $\delta_{irr}$ the fibration $\pi_{1}\circ\phi$ factorizes through $\pi_{1}$ (recall that by Remark \ref{un} on $\overline{\mathcal{M}}_{1,2}$ the forgetful morphisms induce the same fibration). So we have the following commutative diagram:
  \[
  \begin{tikzpicture}[xscale=2.3,yscale=-1.2]
    \node (A0_0) at (0, 0) {$\overline{\mathcal{M}}_{1,2}$};
    \node (A0_1) at (1, 0) {$\overline{\mathcal{M}}_{1,2}$};
    \node (A1_0) at (0, 1) {$\overline{\mathcal{M}}_{1,1}$};
    \node (A1_1) at (1, 1) {$\overline{\mathcal{M}}_{1,1}$};
    \path (A0_0) edge [->]node [auto] {$\scriptstyle{\phi}$} (A0_1);
    \path (A0_0) edge [->]node [auto,swap] {$\scriptstyle{\pi_{1}}$} (A1_0);
    \path (A0_1) edge [->]node [auto] {$\scriptstyle{\pi_{1}}$} (A1_1);
    \path (A1_0) edge [->]node [auto] {$\scriptstyle{\overline{\phi}}$} (A1_1);
  \end{tikzpicture}
  \]
Let $[C,p]\in \overline{\mathcal{M}}_{1,1}$ be a general point and let $[C^{'},p^{'}] = \overline{\phi}([C,p])$ be its image. Then $\alpha : = \phi_{|\pi_{1}^{-1}([C,p])}$ defines an isomorphism between $C$ and $C^{'}$. If $q^{'} = \alpha(p)$ then there exists an automorphism $\tau^{'}$ of $C^{'}$ mapping $q^{'}$ to $p^{'}$. So $\tau^{'}\circ\alpha$ is an isomorphism between $C$ and $C^{'}$ mapping $p$ to $p^{'}$. This means that $[C,p] = [C^{'},p^{'}]$, $\overline{\phi}$ is the identity and $\phi$ restricts to an automorphism of the fiber of $\pi_{1}$, furthermore by Lemma \ref{delta02} has to preserve the boundary divisor $\delta_{0,2}$. The general fiber of $\pi_{1}$ is a general elliptic curve, so it has only two automorphisms. Clearly both these automorphisms act trivially on $\overline{\mathcal{M}}_{1,2}$, so $\phi = Id_{\overline{\mathcal{M}}_{1,2}}$.
\end{proof}

\subsubsection*{Acknowledgements}
I thank \textit{Massimiliano Mella} for many helpful comments, \textit{Barbara Fantechi} for useful discussions and suggestions. Finally, I would like to thank \textit{Mattia Talpo} and \textit{Fabio Tonini} for pointing me out \cite{ACV} and for useful discussions on automorphisms of moduli stacks.

\end{document}